\documentclass{amsart}
\usepackage[utf8]{inputenc}
\usepackage{amsmath}
\usepackage{amssymb}
\usepackage{amsthm}
\usepackage{color}
\usepackage{graphicx}
\usepackage{physics}
\usepackage{tikz-cd}
\usepackage{comment}
\usepackage{soul} 
\usepackage[colorinlistoftodos]{todonotes}
\usepackage[colorlinks=true, allcolors=blue]{hyperref}
\usepackage{subfigure}
\usepackage{subcaption}
\usepackage[arrow,matrix,curve,cmtip,ps]{xy}

\usepackage{mathrsfs}

\makeatletter
\let\@wraptoccontribs\wraptoccontribs
\makeatother

\newcommand{\cl}[1]{\mathcal{#1}}

\newcommand{\bb}[1]{\mathbb{#1}}

\newtheorem{theorem}{Theorem}[section]
\newtheorem{lemma}[theorem]{Lemma}
\newtheorem{proposition}[theorem]{Proposition}

\newtheorem{problem}[theorem]{Problem}

\newtheorem{innertheorem}{Theorem}
\newenvironment{theoremA}[1]{%
  \renewcommand\theinnertheorem{#1}%
  \innertheorem
}{%
  \endinnertheorem
}

\def\bb{\mathbb}
\def\cc{\mathcal}

\DeclareMathOperator{\cA}{{\mathcal A}}

\renewcommand{\min}{min}

\usepackage[a4paper,top=3cm,bottom=2cm,left=3cm,right=3cm,marginparwidth=1.75cm]{geometry}

\title{On C*-nuclearity of graph operator systems}

\author{Adam Dor-On}
\address{Adam Dor-On \\ University of Haifa}
\email{adoron.math@gmail.com}

\author{Travis Russell}
\address{Travis Russell \\ Texas Christian University}
\email{travis.b.russell@tcu.edu}

\contrib[with an appendix by]{Ali Kavruk}
\address{Ali Kavruk}
\email{samilkavruk@gmail.com}

\thanks{This work was supported by the American Institute of Mathematics through the AIM SQuaRE workshop ``Approximation properties for operator systems and matrix convex sets.'' A. Dor-On and T. Russell were partially supported by the BSF startup grant no. 202416. A. Dor-On was partially supported by an NSF-BSF grant no. 2350543 / 2023695 (respectively). T. Russell was also supported by a grant from the TCU Research and Creative Activities Fund.}

\usepackage{natbib}
\usepackage{graphicx}

\begin{document}

\begin{abstract}
We classify the set of finite simple undirected graphs whose corresponding graph operator system is C*-nuclear. We prove that the operator system associated to the graph is C*-nuclear if and only if the graph is 2-chordal, as defined by McCullough. In particular, we present the first known example of a non-chordal graph whose graph operator system is C*-nuclear, as well as the first known examples of graph operator systems which are not C*-nuclear.
\end{abstract}

\maketitle

\section{Introduction} \label{sec: Intro}

The \textit{positive matrix completion problem} asks whether or not a partially defined complex matrix can be completed to a positive definite matrix. A partially defined $n \times n$ matrix is associated with a graph $G=(V,E)$ where $V = \{1,2,\dots,n\}$ and $(i,j) \in E$ whenever the $(i,j)$-entry of the matrix is defined. If a partially defined matrix has a positive definite completion, then every completely defined principle submatrix must be positive definite. In \cite{Grone84}, it was shown that for a fixed graph $G$, every associated partially defined matrix with positive diagonal entries and positive definite complete submatrices has a positive definite completion if and only if $G$ is chordal. This result was generalized to partially-defined operator-valued matrices in \cite{PPS89} and to partially-defined C*-algebra-valued matrices in \cite{PaulsenRodman}.

The main result of this paper can be recast as a relative version of the C*-algebra-valued matrix completion problem of \cite{PaulsenRodman}. Let $\mathcal{A}$ be a C*-algebra and $G=(V,E)$ a graph. A partially defined matrix $(a_{i,j})_{(i,j) \in E}$ with entries in a C*-algebra $\mathcal{A}$ is \text{$\mathcal{A}$-completable} if it has an $\mathcal{A}$-valued positive definite completion $(a_{ij}) \in M_n(\mathcal{A})$. In this paper, we will identify precisely those graphs $G=(V,E)$ such that for every C*-algebra $\mathcal{A} \subseteq B(\mathcal{H})$, if the partially defined matrix $(a_{i,j})_{(i,j) \in E}$ has entries in $\mathcal{A}$ and is $B(\mathcal{H})$-completable, then it is $\mathcal{A}$-completable. Our main result is the following theorem:

\begin{theoremA}{A} \label{thm: main result intro}
A graph $G=(V,E)$ is 2-chordal (in the sense of \cite{McCullough93}) if and only if it has the property that for every C*-algebra $\mathcal{A} \subseteq B(\mathcal{H})$, every partially defined $\mathcal{A}$-valued matrix $(a_{i,j})_{(i,j) \in E}$ which is $B(\mathcal{H})$-completable is also $\mathcal{A}$-completable.
\end{theoremA}

The relative positive completion problem above is motivated by the tensor theory of operator systems developed in \cite{KPTT1} and \cite{KPTT2} by Kavruk, Paulsen, Todorov and Tomforde. An operator system $\mathcal{S}$ is said to be \textit{nuclear} if $\mathcal{S} \otimes_{min} \mathcal{T}$ is completely order isomorphic to $\mathcal{S} \otimes_{max} \mathcal{T}$ for every operator system $\mathcal{T}$, where $\otimes_{min}$ and $\otimes_{max}$ denote the minimal and maximal operator system tensor product structures. An operator system is called \textit{C*-nuclear} if $\mathcal{S} \otimes_{min} \mathcal{A}$ is completely order isomorphic to $\mathcal{S} \otimes_{max} \mathcal{A}$ for every C*-algebra $\mathcal{A}$. In \cite{KPTT1}, a C*-nuclear yet non-nuclear example of an operator system was provided. This example comes from the class of \textit{graph operator systems} associated to simple undirected graphs. Given a simple finite undirected graph $G$ with $n$ vertices, the associated graph operator system $\mathcal{S}_G \subseteq M_n$ is the linear span of the set of matrix units $E_{ij}$ such that $(i,j)$ is an edge of $G$. It was shown in \cite{AKRSx} that a graph operator system $\mathcal{S}_G$ is C*-nuclear if and only if the relative C*-algebra-valued completion problem above has a positive solution for the graph $G$. Since partially defined C*-algebra-valued matrices associated to chordal graphs do have C*-algebra-valued completions by \cite{PaulsenRodman}, every chordal graph is C*-nuclear, as shown in \cite{KPTT1}. On the other hand, Han and Paulsen show in \cite{HP} that a finite-dimensional operator system is nucelar if and only if it is completely order isomorphic to a C*-algebra. Hence the seven-dimensional operator system associated to the three-vertex path $P_3$ is C*-nuclear (since $P_3$ is chordal) but not nuclear since it is not isomorphic to any seven-dimensional C*-algebra.

The main result of this paper, Theorem \ref{thm: main result intro}, is a complete classification of those graphs $G$ for which the associated operator system $\mathcal{S}_G$ is C*-nuclear. Although every graph operator system associated to a chordal graph is known to be C*-nuclear, no examples have appeared in the literature of non-C*-nuclear graph operator systems or of C*-nuclear graph operator systems associated to non-chordal graphs (c.f. \cite[Question 3.7.11]{GuptaEtAl}). We provide the first example in the literature of a non-chordal graph whose graph operator system is C*-nuclear: 
\begin{theoremA}{B} \label{intro thm: c_4 C*-nucl}
    The graph operator system $\mathcal{S}_{C_4}$ associated to the 4-cycle $C_4$ is C*-nuclear.
\end{theoremA}
\noindent From this result, we are able to conclude that every 2-chordal graph is C*-nuclear. This implication follows from a characterization of 2-chordal graphs due to Laurent \cite{Laurent01} (refining prior work of McCullough \cite{McCullough93}). Laurent's characterization shows that every 2-chordal graph can be constructed from the 4-cycle $C_4$ by a sequence of elementary graph operations. We show that each of these operations preserves the C*-nuclearity of the associated graph operator system:
\begin{theoremA}{C} \label{intro thm: permanance}
    Suppose $G$ is a graph and $\mathcal{S}_G$ is C*-nuclear. Then $\mathcal{S}_H$ is C*-nuclear whenever $H$ is obatined from $G$ by a sequence of any of the following graph operations:
    \begin{enumerate}
        \item Joining with a complete graph.
        \item Taking a clique sum with a graph $G'$ for which $\mathcal{S}_{G'}$ is C*-nuclear.
        \item Replacing vertices with cliques (clique blow-up).
        \item Deleting vertices.
        \item Contracting edges.
    \end{enumerate}
\end{theoremA}

To prove the other direction of our main result, we show that whenever $G$ is not 2-chordal, the associated operator system $\mathcal{S}_G$ is not C*-nuclear. The family of 2-chordal graphs is defined by McCullough as those graphs which do not contain any of the six ``forbidden subgraphs'' depicted in Figure \ref{Fig: forbidden subgraphs} as an induced subgraph (i.e. a graph obtained by deleting vertices). Since C*-nuclearity is preserved by vertex deletion, Theorem \ref{intro thm: permanance} implies that it is enough to show each of these six forbidden subgraphs are associated to graph operator systems which are not C*-nuclear. We first show that two of these graphs, $K_{2,3}$ and $2K_1 + P_4$, are associated to non-C*-nuclear graph operator systems by leveraging a result of Kavruk \cite{Kavrukw}: a von Neumann algebra $\mathcal{N} \subseteq B(\mathcal{H})$ is injective if and only if for each integer $d$, whenever $x_1, x_2, y_1, y_2, y_3 \in M_d(\mathcal{N})$ satisfy $x_i < b < y_j$ for all indices $i$ and $j$ and an element $b\in M_d(B(\mathcal{H}))$, then there exists a ``Riesz interpolant'' $a \in M_d(\mathcal{N})$ such that $x_i < a < y_j$ for all indices $i$ and $j$. Using this property, we construct partially defined matrices over a non-injective von Neumann algebra $\mathcal{N}$ for which the existence of a $\mathcal{N}$-valued positive completion implies the existence of a Riesz interpolant. Since non-injectivity of $\mathcal{N}$ implies the existence of $x_1, x_2, y_1, y_2, y_3 \in \mathcal{N}$ which lack a Riesz interpolant in $\mathcal{N}$ when one exists in $B(\mathcal{H})$, we conclude:
\begin{theoremA}{D} \label{intro thm: K_23 and 2K_1+P_4}
    The operator systems $\mathcal{S}_{K_{2,3}}$ and $\mathcal{S}_{2K_1 + P_4}$ are not C*-nuclear.
\end{theoremA}

Three of the remaining four ``forbidden subgraphs'' follow from Theorem \ref{intro thm: K_23 and 2K_1+P_4} by combinatorial arguments: we show that the graph $2K_1 + P_4$ can be obtained from each of the graphs $C_5$, $\overline{C_6}$ and $\overline{P_6}$ by a sequence of graph operations described in Theorem \ref{intro thm: permanance}. It follows that none of these graphs can be C*-nuclear since $2K_1 + P_4$ is not C*-nuclear. For the final graph, $K_{2,2,2}$, we show that the the ``Namioka-Phelps test space'' $\mathcal{W}_{2,3}$ is a retract of the associated operator system $\mathcal{S}_{K_{2,2,2}}$. The former was shown by Kavruk \cite{KavrukNP} to be non-C*-nuclear, so that $\mathcal{S}_{K_{2,2,2}}$ is not C*-nuclear. From these results we rule out the remaining forbidden subgraphs, which allows us to conclude the other direction of Theorem \ref{thm: main result intro}.

\begin{theoremA}{E} \label{intro thm: other forbidden graphs}
    If $G$ is one of the six forbidden subgraphs, then $\mathcal{S}_G$ is not C*-nuclear.
\end{theoremA}

Our paper is organized as follows. In Section \ref{Sec: Prelims}, we discuss preliminary background concerning graphs, operator completion techniques, and operator systems. In Section \ref{sec: Permanence}, we prove that certain combinatorial graph operations preserve the C*-nuclearity of the associated graph operator system. In Section \ref{sec: 2-chordal implies C*-nuclear} we prove that the graph $C_4$ is C*-nuclear, from whence we conclude that every 2-chordal graph is C*-nuclear. In Section \ref{sec: K_2,3 and 2K_1+P_4 not C*-nuc} we prove Theorem \ref{intro thm: K_23 and 2K_1+P_4}. Finally, in Section \ref{sec: not 2-chordal implies not C*-nuclear}, we establish the failure of C*-nuclearity of the remaining forbidden subgraph operator systems, completing the main result.

\vspace{6pt}

\textbf{Acknowledgments.}
We thank Vern Paulsen for sharing open problems concerning graph operator systems with the community and for helpful conversations at the Banff International Research Station workshop ``Operator Systems and their Applications'' and at the Mittag-Leffler Institute program ``Operator algebras and quantum information theory.'' We thank Roy Araiza, Matt Kennedy, Larissa Kroell and Thomas Sinclair for preliminary discussions concerning these problems.

\section{Preliminaries} \label{Sec: Prelims}

In this section, we review definitions, notations, and fundamental results concerning graphs, positive matrix completion problems, and operator systems which will be used throughout the paper.

All graphs in this paper will be finite, simple and undirected. Therefore we define a graph $G$ to be a pair $(V,E)$ where $V$ is a finite set of vertices and $E \subseteq V \times V$ is symmetric, i.e. $(i,j) \in E$ implies that $(j,i) \in E$, and reflexive, i.e. $(i,i) \in E$ for all $i \in V$. We sometimes write $V(G)$ and $E(G)$ in place of $V$ and $E$, respectively, when more specificity is needed. Given a pair of graphs $G$ and $H$, we say that $G$ is \textit{isomorphic} to $H$ if there exists a bijection $\pi: V(G) \to V(H)$ such that $(\pi(i),\pi(j)) \in E(H)$ if and only if $(i,j) \in E(G)$. 

Given a graph $G$ and a subset $U \subseteq V(G)$, the \textit{subgraph $G(U)$ induced by $U$} is the graph with vertex set $U$ and edge set \[ E(U) = \{ (i,j) \in E(G) : i,j \in U \}. \] Equivalently $G(U)$ is the graph obtained by deleting every vertex of $V(G) \setminus U$ as well as every edge connected to a vertex in $V(G) \setminus U$. We say that a graph $H$ is an \textit{induced subgraph} of a graph $G$ if $H$ is isomorphic to $G(U)$ for some subset $U \subseteq V(G)$. An \emph{edge contraction} of a graph $G$ is the graph obtained by replacing a pair of vertices $i,j \in V(G)$ satisfying $(i,j) \in E(G)$ with a new vertex $\{i,j\}$ and replacing each edge connected to $i$ and $j$ with edges connected to a new vertex $\{i,j\}$ (e.g. if $(i,k) \in E(G)$ for $k\neq j$, then $(\{i,j\},k)$ is an edge in the new edge-contracted graph). A \emph{vertex deletion} of a graph $G$ is the graph obtained by removing a fixed vertex $x \in V(G)$ from the vertex set and deleting every edge connected to the removed vertex $x$. Thus, for a subset $U\subset V(G)$, the induced subgraph on $U$ is obtained by performing successive vertex deletions of the vertices in $V(G) \setminus U$. A graph $H$ is called an \textit{induced minor} of a graph $G$ if it is isomorphic to a graph obtained from $G$ by a sequence of \textit{edge contractions} and \textit{vertex deletions}.  Some examples of induced minors of a given graph are depicted in Figure \ref{fig: subgraphs and ind minors}.

\begin{figure*}[h!]
    \centering
    \begin{subfigure}
    \centering
    {
    \begin{tikzpicture}[
      node distance=2cm,
      thick,
      main node/.style={circle, draw, fill=black, minimum size=10pt, inner sep=0pt},
      scale=1.0
    ]
    
      \node[main node, fill=blue] (1) at (60:1.5) {};
      \node[main node, fill=blue] (2) at (120:1.5) {};
      \node[main node, fill=blue] (3) at (180:1.5) {};
      \node[main node, fill=blue] (4) at (240:1.5) {};
      \node[main node, fill=blue] (5) at (300:1.5) {};
      \node[main node, fill=blue] (6) at (0:1.5) {};
      \node[main node, fill=blue]
      (7) at (0,0) {};
    
      \draw[black, thick] (1) -- (2);
      \draw[black, thick] (2) -- (3);
      \draw[black, thick] (3) -- (4);
      \draw[black, thick] (4) -- (5);
      \draw[black, thick] (5) -- (6);
      \draw[black, thick] (6) -- (1);
      \draw[black, thick] (1) -- (7);
      \draw[black, thick] (2) -- (7);
      \draw[black, thick] (3) -- (7);
      \draw[black, thick] (4) -- (7);
      \draw[black, thick] (5) -- (7);
      \draw[black, thick] (6) -- (7);
    
      \node at (0,-2.2) {A graph $G$};

    \end{tikzpicture}
    }
    \end{subfigure}%
    ~ 
    \begin{subfigure}
    \centering
    {
    \begin{tikzpicture}[
      node distance=2cm,
      thick,
      main node/.style={circle, draw, fill=black, minimum size=10pt, inner sep=0pt},
      scale=1.0
    ]
    
      \node[main node, fill=red] 
      (1) at (60:1.5) {};
      \node[main node, fill=red] 
      (2) at (120:1.5) {};
      \node[main node, fill=red] 
      (3) at (180:1.5) {};
      \node[main node, fill=red] 
      (4) at (240:1.5) {};
      \node[main node, fill=red] 
      (5) at (300:1.5) {};
      \node[main node, fill=red] 
      (6) at (0:1.5) {};
    
      \draw[black, thick] (1) -- (2);
      \draw[black, thick] (2) -- (3);
      \draw[black, thick] (3) -- (4);
      \draw[black, thick] (4) -- (5);
      \draw[black, thick] (5) -- (6);
      \draw[black, thick] (6) -- (1);
    
      \node at (0,-2.2) {An induced subgraph of $G$};

    \end{tikzpicture}
    }
    \end{subfigure}%
    \begin{subfigure}
    {
    \begin{tikzpicture}[
      node distance=2cm,
      thick,
      main node/.style={circle, draw, fill=black, minimum size=10pt, inner sep=0pt},
      scale=1.0
    ]
    
      \node[main node, fill=green] 
      (1) at (45:1.5) {};
      \node[main node, fill=green] 
      (2) at (135:1.5) {};
      \node[main node, fill=green] 
      (3) at (225:1.5) {};
      \node[main node, fill=green] 
      (4) at (315:1.5) {};
    
      \draw[black, thick] (1) -- (2);
      \draw[black, thick] (2) -- (3);
      \draw[black, thick] (3) -- (4);
      \draw[black, thick] (4) -- (1);
    
      \node at (0,-2.2) {An induced minor of $G$};

    \end{tikzpicture}
    }
    \end{subfigure}
    \caption{Subgraphs and induced minors} \label{fig: subgraphs and ind minors}
\end{figure*}
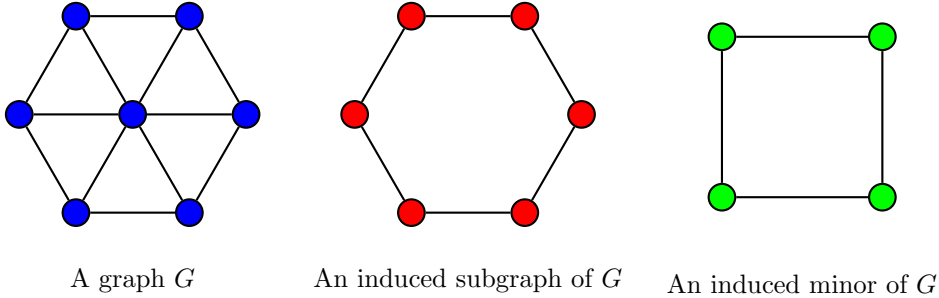

Some fundamental graphs will appear throughout the paper (see Figure \ref{fig: fundamental graphs} for some of these graphs). A graph $K$ is called \textit{complete} if $E(K) = V(K) \times V(K)$. Clearly two complete graphs are isomorphic whenever their vertex sets have the same cardinality. We let $K_n$ denote the complete graph with $n$ vertices. For a graph $G$, a subset $U \subseteq V(G)$ is called a \textit{clique} if the induced subgraph $G(U)$ is complete. We let $C_n$ denote the \textit{$n$-cycle} with vertex set $V(C_n) = \{0,1,\dots,n-1\} =: [n]$ and edge set 
\[ E(C_n) = \{ \ (i,j) \in [n] \times [n] \ : \ 1 \geq |i-j| \ (\mathrm{mod} \ n) \ \}.\] 
We let $P_n$ denote the \textit{$n$-vertex path} with vertex set $V(P_n) = [n]$ and edge set 
\[
E(P_n) = \{(i,j) : |i-j| \leq 1\}
\] 

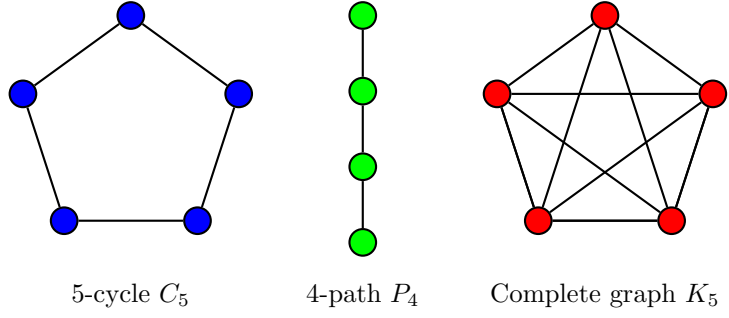
\begin{figure*}[h!]
    \centering
    \begin{subfigure}
    \centering
    {
    \begin{tikzpicture}[
      node distance=2cm,
      thick,
      main node/.style={circle, draw, fill=black, minimum size=10pt, inner sep=0pt},
      scale=1.0
    ]
    
      \node[main node, fill=blue] (1) at (90:1.5) {};
      \node[main node, fill=blue] (2) at (18:1.5) {};
      \node[main node, fill=blue] (3) at (306:1.5) {};
      \node[main node, fill=blue] (4) at (234:1.5) {};
      \node[main node, fill=blue] (5) at (162:1.5) {};
    
      \draw[black, thick] (1) -- (2);
      \draw[black, thick] (2) -- (3);
      \draw[black, thick] (3) -- (4);
      \draw[black, thick] (4) -- (5);
      \draw[black, thick] (5) -- (1);
    
      \node at (0,-2.2) {5-cycle $C_5$};

    \end{tikzpicture}
    }
    \end{subfigure}%
    ~ 
    \begin{subfigure}
    {
    \begin{tikzpicture}[
      node distance=2cm,
      thick,
      main node/.style={circle, draw, fill=black, minimum size=10pt, inner sep=0pt},
      scale=1.0
    ]
    
      \node[main node, fill=green] (1) at (0,1.5) {};
      \node[main node, fill=green] (2) at (0,0.5) {};
      \node[main node, fill=green] (3) at (0,-0.5) {};
      \node[main node, fill=green] (4) at (0,-1.5) {};

      \draw[black, thick] (1) -- (2);
      \draw[black, thick] (2) -- (3);
      \draw[black, thick] (3) -- (4);
    
      \node at (0,-2.2) {4-path $P_4$};

    \end{tikzpicture}
    }
    \end{subfigure}
    ~ 
    \begin{subfigure}
    {
    \begin{tikzpicture}[
      node distance=2cm,
      thick,
      main node/.style={circle, draw, fill=black, minimum size=10pt, inner sep=0pt},
      scale=1.0
    ]
    
      \node[main node, fill=red] 
      (1) at (90:1.5) {};
      \node[main node, fill=red] 
      (2) at (18:1.5) {};
      \node[main node, fill=red] 
      (3) at (306:1.5) {};
      \node[main node, fill=red] 
      (4) at (234:1.5) {};
      \node[main node, fill=red] 
      (5) at (162:1.5) {};
    
      \draw[black, thick] (1) -- (2);
      \draw[black, thick] (2) -- (3);
      \draw[black, thick] (3) -- (4);
      \draw[black, thick] (3) -- (2);
      \draw[black, thick] (3) -- (1);
      \draw[black, thick] (4) -- (5);
      \draw[black, thick] (4) -- (3);
      \draw[black, thick] (4) -- (2);
      \draw[black, thick] (4) -- (1);
      \draw[black, thick] (5) -- (1);
      \draw[black, thick] (5) -- (2);
      \draw[black, thick] (5) -- (3);
      \draw[black, thick] (5) -- (4);
    
      \node at (0,-2.2) {Complete graph $K_5$};

    \end{tikzpicture}
    }
    \end{subfigure}
    \caption{Some fundamental graphs} \label{fig: fundamental graphs}
\end{figure*}

We will frequently construct new graphs from a given graph or a set of graphs by means of certain graph operations. Given a graph $G$, we let $\overline{G}$ denote the \textit{complement} of $G$ with vertex set $V(G)$ and edge set $E(\overline{G}) = \{ \ (i,i) \ : \ \in i \in V(G) \ \} \sqcup \{ \ (i,j) \ : \ i \neq j, \ (i,j) \notin E(G) \ \}$. For a fixed graph $G$ and an integer $n \in \mathbb{N}$, we let $n G$ denote the disjoint union of $n$ copies of $G$ with no edges connecting the distinct copies of $G$. For instance, $3K_1$ is the graph with 3 vertices and no edges. Given two graphs $G$ and $H$, we let $G+H$ denote the \textit{join} of $G$ and $H$ with vertex set $V(G+H) = V(G) \sqcup V(H)$ the disjoint union of their vertex sets, and edge set
\[ E(G+H) = E(G) \cup E(H) \cup \{(i,j), (j,i) : i \in V(G), j \in V(H)\}. \]
An important example of a graph join is the \textit{$n$-wheel} defined by $W_n = K_1 + C_n$ (see Figure \ref{fig: Graph operations} for an example). We define $G_1 + G_2 + \dots + G_n$ to be the graph consisting of the disjoint union of $G_1, \dots G_n$ with additional edges connecting each vertex of $G_i$ to each vertex of $G_j$ for $i \neq j$. Thus, for instance, we define $K_{n_1, n_2, \dots, n_m}$ to be the graph $n_1 K_1 + n_2 K_1 + \dots + n_m K_1$ (the $m$-partite graph with parts $n_i K_1$ of $n_i$ disjoint vertices).  Given two graphs $G = (V_1, E_1)$ and $H=(V_2,E_2)$ for which there is a common induced clique subgraph $K$ in both $G$ and $H$, we define the \textit{clique sum} $G \vee_K H$ to be the graph with vertex set $V(G \vee_K H) = V_1 \cup V_2$ and edge set $E(G \vee_K H) = E_1 \cup E_2$. Finally, given a graph $G$ with vertex set $V(G) = [m]$ and an integer $n \in \mathbb{N}$, we define the $n$-\textit{clique blowup} of $G$ to be the graph $G^{(n)}$ with vertex set
\[ V(G^{(n)}) = \{0_0, 0_1, \dots, 0_{n-1}, 1_0, 1_1, \dots, 1_{n-1}, \dots, (m-1)_0, (m-1)_2, \dots, (m-1)_{n-1} \} \]
and edge set
\[ 
E(G^{(n)}) = \{ \ (i_j, i_k) \ : \ j \neq k, \ i \in V(G) \ \} \cup \{ \ (i_k, j_l) \ : \ (i,j) \in E(G), \ k,l \in [n] \ \}. 
\]
In other words, $G^{(n)}$ is the graph obtained by replacing each vertex of $G$ with a $n$-clique and connecting vertices of separate cliques if the original vertices were connected in $G$.

\begin{figure*}[h!] 
    \centering
    \begin{subfigure}
    {
    \begin{tikzpicture}[
      node distance=2cm,
      thick,
      main node/.style={circle, draw, fill=black, minimum size=10pt, inner sep=0pt},
      scale=1.0
    ]
    
      \node[main node, fill=blue] (1) at (60:1.5) {};
      \node[main node, fill=blue] (2) at (120:1.5) {};
      \node[main node, fill=blue] (3) at (180:1.5) {};
      \node[main node, fill=blue] (4) at (240:1.5) {};
      \node[main node, fill=blue] (5) at (300:1.5) {};
      \node[main node, fill=blue]
      (6) at (0:1.5) {};
    
      \draw[black, thick] (3) -- (1);
      \draw[black, thick] (1) -- (5);
      \draw[black, thick] (4) -- (2);
      \draw[black, thick] (4) -- (1);
      \draw[black, thick] (5) -- (2);
      \draw[black, thick] (5) -- (3);
      \draw[black, thick] (6) -- (2);
      \draw[black, thick] (6) -- (3);
      \draw[black, thick] (6) -- (4);
      \draw[red, dotted] (1) -- (2);
      \draw[red, dotted] (2) -- (3);
      \draw[red, dotted] (3) -- (4);
      \draw[red, dotted] (4) -- (5);
      \draw[red, dotted] (5) -- (6);
      \draw[red, dotted] (6) -- (1);
    
      \node at (0,-2.2) {(a) Graph complement $\overline{C_n}$, $n \geq 5$};

    \end{tikzpicture}
    }
    \end{subfigure}
    \begin{subfigure}
    {
    \begin{tikzpicture}[
      node distance=2cm,
      thick,
      main node/.style={circle, draw, fill=black, minimum size=10pt, inner sep=0pt},
      scale=1.0
    ]
    
      \node[main node, fill=red] (1) at (150:1.5) {};
      \node[main node, fill=red] (2) at (210:1.5) {};
      \node[main node, fill=green] (3) at (60:1.5) {};
      \node[main node, fill=green] (4) at (0:0.75) {};
      \node[main node, fill=green] (5) at (300:1.5) {};
    
      \draw[black, thick] (1) -- (3);
      \draw[black, thick] (1) -- (4);
      \draw[black, thick] (1) -- (5);
      \draw[black, thick] (2) -- (3);
      \draw[black, thick] (2) -- (4);
      \draw[black, thick] (2) -- (5);
    
      \node at (0,-2.2) {(b) Join $2K_1 + 3K_1 = K_{2,3}$};

    \end{tikzpicture}
    }
    \end{subfigure}
    \begin{subfigure}
        {
        \begin{tikzpicture}[
      node distance=2cm,
      thick,
      main node/.style={circle, draw, fill=black, minimum size=10pt, inner sep=0pt},
      scale=1.0
    ]
    
      \node[main node, fill=yellow] 
      (1) at (45:1.5) {};
      \node[main node, fill=yellow] 
      (2) at (135:1.5) {};
      \node[main node, fill=yellow] 
      (3) at (225:1.5) {};
      \node[main node, fill=yellow] 
      (4) at (315:1.5) {};
      \node[main node, fill=yellow]
      (5) at (0,0) {};
    
      \draw[black, thick] (1) -- (2);
      \draw[black, thick] (2) -- (3);
      \draw[black, thick] (3) -- (4);
      \draw[black, thick] (4) -- (1);
      \draw[black, thick] (5) -- (1);
      \draw[black, thick] (5) -- (2);
      \draw[black, thick] (5) -- (3);
      \draw[black, thick] (5) -- (4);
    
      \node at (0,-2.2) {(c) The 4-wheel $W_4 = K_1 + C_4$.};

    \end{tikzpicture}
        }
    \end{subfigure} 
    \begin{subfigure}
    {
    \begin{tikzpicture}[
      node distance=2cm,
      thick,
      main node/.style={circle, draw, fill=black, minimum size=10pt, inner sep=0pt},
      scale=1.0
    ]
    
      \node[main node, fill=blue] (1) at (120:1.5) {};
      \node[main node, fill=blue] (2) at (180:1.5) {};
      \node[main node, fill=blue] (3) at (240:1.5) {};
      \node[main node, fill=green] (4) at (300:1.5) {};
      \node[main node, fill=green] (5) at (0:1.5) {};
      \node[main node, fill=green]
      (6) at (60:1.5) {};
      \node[main node, fill=red]
      (7) at (0,0.5) {};
      \node[main node, fill=red]
      (8) at (0,-0.5) {};
    
      \draw[black, thick] (1) -- (2);
      \draw[black, thick] (1) -- (3);
      \draw[black, thick] (1) -- (7);
      \draw[black, thick] (1) -- (8);
      \draw[black, thick] (2) -- (3);
      \draw[black, thick] (2) -- (7);
      \draw[black, thick] (2) -- (8);
      \draw[black, thick] (3) -- (7);
      \draw[black, thick] (3) -- (8);
      \draw[black, thick] (4) -- (5);
      \draw[black, thick] (4) -- (8);
      \draw[black, thick] (5) -- (6);
      \draw[black, thick] (6) -- (7);
      \draw[black, thick] (7) -- (8);
    
      \node at (0,-2.2) {(d) Clique sum of $K_5$ and $C_5$};

    \end{tikzpicture}
    }
    \end{subfigure}
    \begin{subfigure}
    {
    \begin{tikzpicture}[
      node distance=2cm,
      thick,
      main node/.style={circle, draw, fill=black, minimum size=10pt, inner sep=0pt},
      scale=1.0
    ]
    
      \node[main node, fill=blue] (1) at (120:1.5) {};
      \node[main node, fill=blue] (2) at (150:1.5) {};
      \node[main node, fill=green] (3) at (210:1.5) {};
      \node[main node, fill=green] (4) at (240:1.5) {};
      \node[main node, fill=red] (5) at (300:1.5) {};
      \node[main node, fill=red]
      (6) at (330:1.5) {};
      \node[main node, fill=yellow]
      (7) at (30:1.5) {};
      \node[main node, fill=yellow]
      (8) at (60:1.5) {};
    
      \draw[black, thick] (1) -- (2);
      \draw[black, thick] (1) -- (3);
      \draw[black, thick] (1) -- (4);
      \draw[black, thick] (2) -- (3);
      \draw[black, thick] (2) -- (4);
      \draw[black, thick] (3) -- (4);
      \draw[black, thick] (3) -- (5);
      \draw[black, thick] (3) -- (6);
      \draw[black, thick] (4) -- (5);
      \draw[black, thick] (4) -- (6);
      \draw[black, thick] (5) -- (6);
      \draw[black, thick] (5) -- (7);
      \draw[black, thick] (5) -- (8);
      \draw[black, thick] (6) -- (7);
      \draw[black, thick] (6) -- (8);
      \draw[black, thick] (7) -- (8);
      \draw[black, thick] (7) -- (1);
      \draw[black, thick] (7) -- (2);
      \draw[black, thick] (8) -- (1);
      \draw[black, thick] (8) -- (2);
    
      \node at (0,-2.2) {(e) Clique blowup $C_4^{(2)}$};

    \end{tikzpicture}
    }
    \end{subfigure}
    \caption{Graph operations} \label{fig: Graph operations}
\end{figure*}

A graph is said to be \textit{chordal} if it does not contain a cycle $C_n$ for $n \geq 4$ as an induced subgraph. G. A. Dirac \cite{Dirac61} proved that a graph is chordal if and only if it is a clique sum of cliques. Chordal graphs can also be characterized in terms of their \textit{sparsity order}. Given a graph $G$, let $\mathcal{P}_G$ denote the set of all positive semidefinite matrices over $\mathbb{C}$ with sparsity pattern specified by $G$, meaning that there is a 0 in the $(i,j)$ entry of the matrix whenever $(i,j)$ is not an edge of $G$. The sparsity order $\text{ord}(G)$ of $G$ is defined to be the maximum rank of an extreme ray of the cone $\mathcal{P}_G$. Agler, McCullough, Helton and Rodman proved in \cite{AHMR88} that a graph $G$ is chordal if and only if $\text{ord}(G) = 1$.

The graphs $G$ such that $\text{ord}(G) \leq 2$ were characterized by Laurent in \cite[Theorem 13]{Laurent01}. A stronger condition was previously studied by McCullough in \cite{McCullough88} and \cite{McCullough93}. McCullough defines a graph to be \textit{2-chordal} if it does not contain an induced subgraph from the family of forbidden graphs $\mathcal{F}$ as depicted in Figure \ref{fig: Graph operations}. In the next theorem, we refer to the $4$-wheel graph $W_4$ from Figure \ref{fig: Graph operations} (c) and the family of forbidden graphs $\mathcal{F}$ from Figure \ref{Fig: forbidden subgraphs}.

\begin{theorem}[McCullough \cite{McCullough93}, c.f. Theorem 15 of \cite{Laurent01}] \label{thm: McCullough}
    Let $G$ be a graph. The following statements are equivalent:
    \begin{enumerate}
        \item $\text{ord}(G^{(m)}) \leq 2$ for every $m\geq 1$.
        \item $\text{ord}(G^{(2)}) \leq 2$. 
        \item $G$ is a clique sum of induced subgraphs of the graph $W_4^{(n)}$ for $n \geq 1$.
        \item $G$ does not contain an induced subgraph from the family $\mathcal{F}$, i.e. $G$ is 2-chordal.
    \end{enumerate}
\end{theorem}

\vspace{6pt}

\begin{figure*}[h!]
    \centering
    \begin{subfigure}
    {
    \begin{tikzpicture}[
      node distance=2cm,
      thick,
      main node/.style={circle, draw, fill=black, minimum size=10pt, inner sep=0pt},
      scale=1.0
    ]
    
      \node[main node, fill=blue] (1) at (60:1.5) {};
      \node[main node, fill=blue] (2) at (120:1.5) {};
      \node[main node, fill=blue] (3) at (180:1.5) {};
      \node[main node, fill=blue] (4) at (240:1.5) {};
      \node[main node, fill=blue] (5) at (300:1.5) {};
      \node[main node, fill=blue]
      (6) at (0:1.5) {};
    
      \draw[black, thick] (3) -- (1);
      \draw[black, thick] (1) -- (5);
      \draw[black, thick] (4) -- (2);
      \draw[black, thick] (4) -- (1);
      \draw[black, thick] (5) -- (2);
      \draw[black, thick] (5) -- (3);
      \draw[black, thick] (6) -- (2);
      \draw[black, thick] (6) -- (3);
      \draw[black, thick] (6) -- (4);
      \draw[red, dotted] (1) -- (2);
      \draw[red, dotted] (2) -- (3);
      \draw[red, dotted] (3) -- (4);
      \draw[red, dotted] (4) -- (5);
      \draw[red, dotted] (5) -- (6);
      \draw[red, dotted] (6) -- (1);
    
      \node at (0,-2.2) {(a) Graph complement $\overline{C_6}$};

    \end{tikzpicture}
    }
    \end{subfigure}
    \begin{subfigure}
    {
    \begin{tikzpicture}[
      node distance=2cm,
      thick,
      main node/.style={circle, draw, fill=black, minimum size=10pt, inner sep=0pt},
      scale=1.0
    ]
    
      \node[main node, fill=blue] (1) at (60:1.5) {};
      \node[main node, fill=blue] (2) at (120:1.5) {};
      \node[main node, fill=blue] (3) at (180:1.5) {};
      \node[main node, fill=blue] (4) at (240:1.5) {};
      \node[main node, fill=blue] (5) at (300:1.5) {};
      \node[main node, fill=blue]
      (6) at (0:1.5) {};
    
      \draw[black, thick] (3) -- (1);
      \draw[black, thick] (1) -- (5);
      \draw[black, thick] (4) -- (2);
      \draw[black, thick] (4) -- (1);
      \draw[black, thick] (5) -- (2);
      \draw[black, thick] (5) -- (3);
      \draw[black, thick] (6) -- (2);
      \draw[black, thick] (6) -- (3);
      \draw[black, thick] (6) -- (4);
      \draw[red, dotted] (1) -- (2);
      \draw[red, dotted] (2) -- (3);
      \draw[red, dotted] (3) -- (4);
      \draw[black, thick] (4) -- (5);
      \draw[red, dotted] (5) -- (6);
      \draw[red, dotted] (6) -- (1);
    
      \node at (0,-2.2) {(b) Graph complement $\overline{P_6}$};

    \end{tikzpicture}
    }
    \end{subfigure}
    \begin{subfigure}
    \centering
    {
    \begin{tikzpicture}[
      node distance=2cm,
      thick,
      main node/.style={circle, draw, fill=black, minimum size=10pt, inner sep=0pt},
      scale=1.0
    ]
    
      \node[main node, fill=blue] (1) at (90:1.5) {};
      \node[main node, fill=blue] (2) at (18:1.5) {};
      \node[main node, fill=blue] (3) at (306:1.5) {};
      \node[main node, fill=blue] (4) at (234:1.5) {};
      \node[main node, fill=blue] (5) at (162:1.5) {};
    
      \draw[black, thick] (1) -- (2);
      \draw[black, thick] (2) -- (3);
      \draw[black, thick] (3) -- (4);
      \draw[black, thick] (4) -- (5);
      \draw[black, thick] (5) -- (1);
    
      \node at (0,-2.2) {(c) The $n$-cycle $C_n$, $n \geq 5$};

    \end{tikzpicture}
    }
    \end{subfigure}%
    \begin{subfigure}
    {
    \begin{tikzpicture}[
      node distance=2cm,
      thick,
      main node/.style={circle, draw, fill=black, minimum size=10pt, inner sep=0pt},
      scale=1.0
    ]
    
      \node[main node, fill=blue] (1) at (150:1.5) {};
      \node[main node, fill=blue] (2) at (210:1.5) {};
      \node[main node, fill=blue] (3) at (60:1.5) {};
      \node[main node, fill=blue] (4) at (0:0.75) {};
      \node[main node, fill=blue] (5) at (300:1.5) {};
    
      \draw[black, thick] (1) -- (3);
      \draw[black, thick] (1) -- (4);
      \draw[black, thick] (1) -- (5);
      \draw[black, thick] (2) -- (3);
      \draw[black, thick] (2) -- (4);
      \draw[black, thick] (2) -- (5);
    
      \node at (0,-2.2) {(d) The graph $K_{2,3}$};

    \end{tikzpicture}
    }
    \end{subfigure}
    \begin{subfigure}
    {
    \begin{tikzpicture}[
      node distance=2cm,
      thick,
      main node/.style={circle, draw, fill=black, minimum size=10pt, inner sep=0pt},
      scale=1.0
    ]
    
      \node[main node, fill=blue] (1) at (0:1.5) {};
      \node[main node, fill=blue] (2) at (180:1.5) {};
      \node[main node, fill=blue] (3) at (60:1.5) {};
      \node[main node, fill=blue] (4) at (240:1.5) {};
      \node[main node, fill=blue] (5) at (120:1.5) {};
      \node[main node, fill=blue] (6) at (300:1.5) {};
    
      \draw[black, thick] (1) -- (3);
      \draw[black, thick] (1) -- (4);
      \draw[black, thick] (1) -- (5);
      \draw[black, thick] (1) -- (6);
      \draw[black, thick] (2) -- (3);
      \draw[black, thick] (2) -- (4);
      \draw[black, thick] (2) -- (5);
      \draw[black, thick] (2) -- (6);
      \draw[black, thick] (3) -- (5);
      \draw[black, thick] (3) -- (6);
      \draw[black, thick] (4) -- (5);
      \draw[black, thick] (4) -- (6);
    
      \node at (0,-2.2) {(e) The graph $K_{2,2,2}$};

    \end{tikzpicture}
    }
    \end{subfigure}
    \begin{subfigure}
    {
    \begin{tikzpicture}[
      node distance=2cm,
      thick,
      main node/.style={circle, draw, fill=black, minimum size=10pt, inner sep=0pt},
      scale=1.0
    ]
    
      \node[main node, fill=blue] (1) at (0,1.5) {};
      \node[main node, fill=blue] (2) at (0,0.5) {};
      \node[main node, fill=blue] (3) at (0,-0.5) {};
      \node[main node, fill=blue] (4) at (0,-1.5) {};
      \node[main node, fill=blue] (5) at (0:1.5) {};
      \node[main node, fill=blue] (6) at (180:1.5) {};
    
      \draw[black, thick] (1) -- (2);
      \draw[black, thick] (2) -- (3);
      \draw[black, thick] (3) -- (4);
      \draw[black, thick] (5) -- (1);
      \draw[black, thick] (5) -- (2);
      \draw[black, thick] (5) -- (3);
      \draw[black, thick] (5) -- (4);
      \draw[black, thick] (6) -- (1);
      \draw[black, thick] (6) -- (2);
      \draw[black, thick] (6) -- (3);
      \draw[black, thick] (6) -- (4);
    
      \node at (0,-2.2) {(f) The graph $2K_1 + P_4$};

    \end{tikzpicture}
    }
    \end{subfigure}
    \caption{The family of forbidden subgraphs $\mathcal{F}$} \label{Fig: forbidden subgraphs}
\end{figure*}
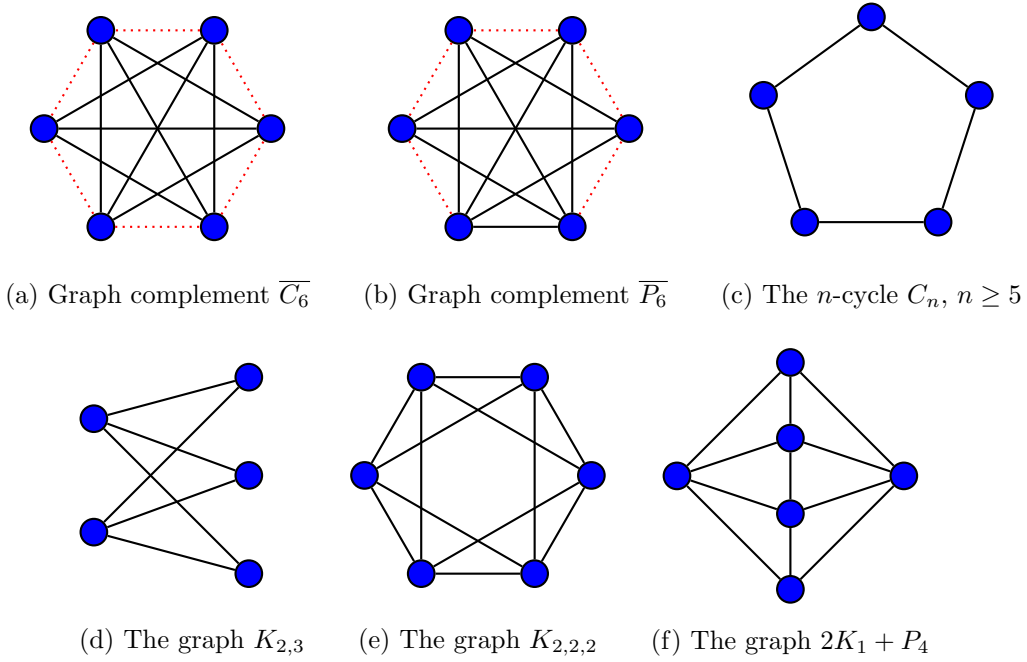

Next, we recall several positive matrix completion problems alluded to in the introduction and review a few operator-theoretic tools that will be used in the paper. We begin with the positive definite matrix completion problem of \cite{Grone84}.

\begin{problem} \label{prob: scalar matrix completion}
    Let $G=(V,E)$ be a graph and $\Lambda = (\lambda_{i,j})_{(i,j) \in E}$ a partially-defined matrix with entries in $\mathbb{C}$. If every complete submatrix of $\Lambda$ is positive definite, does there exist a positive definite completion $\widetilde{\Lambda} \in M_n$?
\end{problem}

\noindent Grone, Horn, S\'a and Wolkowicz proved in \cite{Grone84} that Problem \ref{prob: scalar matrix completion} has a positive solution whenever $G$ is chordal. On the other hand, if $G=(V,E)$ is not chordal, then there exist partially-defined matrices $(\lambda_{i,j})_{(i,j) \in E}$ for which every complete submatrix is positive definite and yet no positive definite completion exists. This result generalizes an earlier result of Dym and Gohberg \cite{DymGohberg81} concerning matrices for which finitely many diagonal bands are specified. Note that here positive definiteness can be replaced with positive \emph{semi}-definiteness by applying the positive definite version of the result to the partially-defined matrix $\Lambda + \epsilon I_n$ for $\epsilon > 0$ and appealing to a compactness argument.

We now consider an operator-valued analog of this problem. Throughout the paper, we let $\mathcal{H}$ denote a Hilbert space and we let $B(\mathcal{H})$ denote the algebra of bounded operators on $\mathcal{H}$. We identify operators in $M_n(B(\cc H))$ with operators on the $n$-fold direct sum $\cc H^{(n)}$ in the standard way.

\begin{problem} \label{prob: operator matrix completion}
    Let $G = (V,E)$ be a graph and $T = (T_{i,j})_{(i,j) \in E}$ be a partially-defined matrix with values in $B(\cc H)$. If every complete submatrix of $T$ is positive semi-definite, does there exist a positive semi-definite completion $\widetilde{T} \in B(\cc H^{(n)})$?
\end{problem}

\noindent In \cite[Theorem 4.3]{PPS89} it was shown that Problem \ref{prob: operator matrix completion} has a positive solution whenever $G$ is chordal. When $G$ is not chordal, counter-examples can be constructed by considering matrices where every entry is a scalar multiple of the identity operator and appealing to the scalar version of the problem. The C*-algebra-valued completion problems were considered in \cite{PaulsenRodman}. Recall that a \textit{C*-algebra} is a norm-closed self-adjoint subalgebra of $B(\cc H)$.

\begin{problem} \label{prob: C* matrix completion}
    Let $G = (V,E)$ be a graph and $A = (a_{i,j})_{(i,j) \in E}$ be a partially-defined matrix with entries in a unital C*-algebra $\cc A$. If every complete submatrix of $A$ is positive definite, does there exist a positive definite completion $\widetilde{A} \in M_n(\cc A)$?
\end{problem}

\noindent In \cite[Theorem 2.1]{PaulsenRodman}, it is shown that Problem \ref{prob: C* matrix completion} has a positive solution whenever $G$ is chordal, and again that counter-examples arise when $G$ is not chordal. In this level of generality, however, positive definiteness (i.e. the requirement that the completed matrix be bounded below by $\epsilon I_n$ for some $\epsilon > 0$) cannot be replaced with positive semi-definiteness. More precisely, \cite[Example 3.1]{PaulsenRodman} provides an example of a $C(\mathbb{T})$-valued partially defined matrix associated to a chordal graph which has no $C(\mathbb{T})$-valued positive semi-definite completion, despite having an $L^{\infty}(\mathbb{T})$-valued positive semi-definite completion. In \cite[Theorem 3.1]{PaulsenRodman}, it is shown that if the C*-algebra $\cc A$ is \textit{injective}, or if $\cc A$ is a \textit{von Neumann algebra} (i.e. $\cc A$ is equal to its double-commutant $\cc A''$), then positive definiteness can be replaced with positive semi-definiteness. More generally, we can say the following about positive semi-definite completions of C*-algebra-valued matrices.

\begin{proposition} \label{prop: psd completion in A''}
    Let $G=(V,E)$ be a graph with $n$ vertices, $\mathcal{A} \subseteq B(\cc H)$ be a unital C*-algebra, and $A = (a_{i,j})_{(i,j) \in E}$ be a partially-defined matrix with entries in $\cc A$. If $A + \epsilon I_n$ has a positive semi-definite completion in $M_n(\cc A)$ for every $\epsilon > 0$, then $A$ has a positive semi-definite completion in $M_n(\cc A'')$.
\end{proposition}

\begin{proof}
    Observe that the set of completions $\{\widetilde{A}_{\epsilon}\}$ of $A + \epsilon I_n$ lie in an ultraweakly compact subset of $\cc A''$. A limit point of this family will therefore be a positive semi-definite completion for $A$.
\end{proof}

We now state the matrix completion problem that is of main interest for this paper.

\begin{problem} \label{prob: main prob}
    Let $G=(V,E)$ be a graph on $n$ vertices and $A = (a_{i,j})_{(i,j) \in E}$ be a partially-defined matrix with values in a unital C*-algebra $\cc A \subseteq B(\cc H)$. If $A + \epsilon I_n$ has a $B(\cc H)$-valued positive definite completion for every $\epsilon > 0$, does $A + \epsilon I_n$ have a $\cc A$-valued positive definite completion for every $\epsilon > 0$?
\end{problem}

\noindent Clearly if a partially-defined matrix $A$ has a positive definite $B(\cc H)$-valued completion, then every complete submatrix of $A$ will be positive definite. Thus the hypothesis of Problem \ref{prob: main prob} is stronger than the hypothesis of Problem \ref{prob: C* matrix completion}. It follows from \cite{PaulsenRodman} that Problem \ref{prob: main prob} has a positive solution whenever $G$ is chordal. The main result of the paper will be to characterize the set of graphs $G$ for which Problem \ref{prob: main prob} has a positive solution. We will see that this class of graphs is strictly larger than the class of chordal graphs. We will apply the following standard Schur complement criteria later in the paper to address matrix completion questions, and their proofs can be found in \cite[Lemma 3.1]{Pnbook} and \cite[Lemma 24]{Oz}.

\begin{lemma} \label{lem: matrix completion and schur}
    Let $\cc H_0, \cc \mathcal{H}_1, \cc \mathcal{H}_2$ be Hilbert spaces. The following statements are true.
    \begin{enumerate}
        \item Suppose that $B \in B(\cc \mathcal{H}_1)$ is positive definite, $A \in B(\cc \mathcal{H}_0)$ and $X \in B(\cc \mathcal{H}_1, \cc \mathcal{H}_0)$. Then
        \[ 
        \begin{pmatrix} A & X \\ X^* & B \end{pmatrix} \in B(\cc \mathcal{H}_0 \oplus \cc \mathcal{H}_1) 
        \]
        is positive semi-definite if and only if $A \geq X B^{-1} X^*$.
        \item Suppose that the operator matrices
        \[ 
        \begin{pmatrix} A & X \\ X^* & B \end{pmatrix} \in B(\mathcal{H}_0 \oplus \mathcal{H}_1) \quad \text{and} \quad \begin{pmatrix} B & Y \\ Y^* & C \end{pmatrix} \in B(\mathcal{H}_1 \oplus \mathcal{H}_2) 
        \] 
        are both positive definite. Then the matrix
        \[ 
        \begin{pmatrix} A & X & Z \\ X^* & B & Y \\ Z^* & Y^* & C \end{pmatrix} \in B(\mathcal{H}_0 \oplus \mathcal{H}_1 \oplus \mathcal{H}_2) 
        \] 
        is positive for $Z = X B^{-1} Y$.
    \end{enumerate}
\end{lemma}

\vspace{6pt}

We conclude with a summary of operator system theory relevant to this work. An \textit{operator system} is a unital self-adjoint subspace of a C*-algebra. Operator systems were abstractly characterized by Choi and Effros \cite{CEinj} as $*$-vector spaces $\cc S$ equipped with a sequence of positive cones $\{M_n(\cc S)^+\}_{n=1}^{\infty}$ in each $n\times n$ matrix amplification $M_n(\cc S)$, called a \textit{matrix ordering}, and with a specified Archimedean order unit $I \in \cc S$. Therefore the relevant morphisms between operator systems are linear maps preserving the unit, involution and the matrix ordering. More precisely, given operator systems $\cc S$ and $\cc T$, a unit-preserving linear map $\varphi: \cc S \to \cc T$ which preserves the matrix cones is called \textit{unital completely positive} (or UCP for short), where the latter condition asks for the maps $\varphi^{(n)}: M_n(\cc S) \to M_n(\cc T)$ to be positive for every $n \in \mathbb{N}$, where $\varphi^{(n)}$ is the map applying $\varphi$ to the entries of each matrix in $M_n(\cc S)$. A unital completely positive map $\varphi: \cc S \to \cc T$ is said to be a \textit{unital complete order embedding} if it is injective and if $\varphi^{-1}$ is completely positive when restricted to the image $\varphi(\cc S)$. We may therefore view operator systems abstractly as matrix-ordered $*$-vector spaces $\cc S$ with an Archimedean order unit for which there exists a unital complete order embedding $\varphi: \cc S \to \cc A$ into some unital C*-algebra $\cc A$.

The tensor theory of operator systems was initiated in \cite{KPTT1} by Kavruk, Paulsen, Todorov and Tomforde. In this paper we will be concerned with three tensor products developed in \cite{KPTT1}, namely the \textit{minimal}, \textit{maximal}, and \textit{commuting} tensor product. For the sake of brevity, we will only summarize the relevant definitions and refer the interested reader to \cite{KPTT1} for further details. Let $\cc S \subseteq B(\cc \mathcal{H}_1)$ and $\cc T \subseteq B(\cc \mathcal{H}_2)$ be operator systems viewed concretely as untial self-adjoint subspaces of bounded operators. Then the minimal tensor product $\cc S \otimes_{min} \cc T$ is the operator system generated by the inclusion of the algebraic tensor product $\cc S \otimes \cc T$ into $B(\mathcal{H}_1 \otimes \mathcal{H}_2)$. This operator system turns out to be independent of the choice of complete order embedding of $\cc S$ and $\cc T$ into $B(\mathcal{H}_1)$ and $B(\mathcal{H}_2)$, respectively. The maximal tensor product $\cc S \otimes_{max} \cc T$ is the operator system with the smallest matrix ordering generated by the positive operators $\{ \ P \otimes Q \  : \  P \in M_n(\cc S), Q \in M_m(\cc T), \ n,m \in \mathbb{N} \ \}$. This operator system endows the vector space $\cc S \otimes \cc T$ with the largest possible operator norm that can be embedded (in an appropriate sense) as a unital subspace of a C*-algebra. When $\cc S$ and $\cc T$ are C*-algebras, the maximal tensor product $\cc S \otimes_{max} \cc T$ coincides with the image of the algebraic tensor product $\cc S \otimes \cc T$ in the C*-algebra maximal tensor product of $\cc S$ and $\cc T$. Finally, the \textit{commuting tensor product} $\cc S \otimes_c \cc T$ is the unique operator system structure on the algebraic tensor product $\cc S \otimes \cc T$ which is universal for unital completely positive maps with commuting ranges. That is, if $\varphi: \cc S \to B(\cc H)$ and $\psi: \cc T \to B(\cc H)$ are unital completely positive maps with commuting ranges, and $\varphi \cdot \psi : \cc S \otimes_c \cc T \to B(\cc H)$ is the linear map defined on elementary tensors by $\varphi \cdot \psi(x \otimes y) = \varphi(x) \psi(y)$, then $\varphi \cdot \psi$ is unital completely positive.

An operator system $\cc S$ is called \textit{nuclear} if $\cc S \otimes_{min} \cc T = \cc S \otimes_{max} \cc T$ for every operator system $\cc T$, and is called \textit{(min,c)-nuclear} if $\cc S \otimes_{min} \cc T = \cc S \otimes_{c} \cc T$ for every operator system $\cc T$. It is called \textit{C*-nuclear} if $\cc S \otimes_{min} \cc A = \cc S \otimes_{max} \cc A$ (equivalently, $\cc S \otimes_{min} \cc A = \cc S \otimes_{c} \cc A$) for every C*-algebra $\cc A$. In \cite{KPTT1} it was shown that an operator system is (min,c)-nuclear if and only if it is C*-nuclear.

We will be primarily concerned with the class of \textit{graph operator systems}. Given a graph $G=(V,E)$ with vertex set $V=[n]$, we define the graph operator system \[ \cc S_G := \text{span} \{ E_{i,j} : (i,j) \in E \} \subseteq M_n \] where $E_{i,j}$ denotes the matrix with a 1 in the $(i,j)$ entry and zeros elsewhere. The study of graph operator systems was initiated in \cite{KPTT1} where the following theorem was proven.

\begin{theorem}
    Suppose that $G$ is a graph. If $G$ is chordal, then $\cc S_G$ is C*-nuclear. On the other hand, there exists a chordal graph $G$ such that $\cc S_G$ is not nuclear.
\end{theorem}

\noindent In particular, it was shown that $\cc S_{P_3}$ is C*-nuclear but not nuclear (where $P_3$ is the 3-path). That $\cc S_{P_3}$ is not nuclear follows from a theorem of Han and Paulsen \cite{HP} that every nuclear finite-dimensional operator system is completely order isomorphic to a C*-algebra.

We will make use of a recent characterization of C*-nuclearity of graph operator systems from \cite{AKRSx}. Namely, if $\cc S_G$ is the graph operator system of a graph $G$, it was shown in \cite[Theorem 3.10]{AKRSx} that $\cc S_G$ is C*-nuclear if and only if Problem \ref{prob: main prob} has a positive solution for the graph $G$.

\section{Permanence of C*-nuclearity} \label{sec: Permanence}

In this section, we show that a variety of combinatorial graph constructions preserve the C*-nuclearity of the corresponding graph operator system. To simplify the terminology, we say that a graph $G$ is \textit{C*-nuclear} if the corresponding operator system $\cc S_G$ is C*-nuclear.

\begin{proposition} \label{prop: G+K_n C*-nuclear}
    Suppose that $G$ is C*-nuclear. Then $G + K_n$ is C*-nuclear for every $n \in \mathbb{N}$.
\end{proposition}

\begin{proof}
    We will show the corresponding matrix completion problem has a solution. Let $\cA \subseteq B(\mathcal{H})$ be a C*-algebra. Suppose the partially defined matrix
    \[ 
    \begin{pmatrix} P_{?} & X \\ X^* & Q \end{pmatrix} + \epsilon I_{d+n} \in M_{d+n}(\cA) 
    \]
    has a positive completion in $B(\mathcal{H})$, where some entries of $P_? \in M_d(\cA)$ are unspecified. Let $P_0 \in M_d(B(\mathcal{H}))$ be the operator that completes the 1-1 block to yield a positive operator. Let $Q_{\epsilon} := Q + \epsilon I_n$. Conjugating by $I_d \oplus Q_{\epsilon}^{-1/2}$, we obtain the positive matrix
    \[ 
    \begin{pmatrix} P_0 + \epsilon I_{d} & X Q_{\epsilon}^{-1/2} \\ Q_{\epsilon}^{-1/2} X^* & I_n \end{pmatrix}. 
    \] 
    By Lemma \ref{lem: matrix completion and schur}, $X Q_{\epsilon}^{-1} X^* \leq P_0 + \epsilon I_{d}$. It follows that $P_0 + \epsilon I_d - X Q_{\epsilon}^{-1} X^*$ is a positive operator completion of the partially defined matrix $R \in M_d(\cA)$ obtained by forgetting the entries of $P_0 + \epsilon I_d - X Q_{\epsilon}^{-1} X^*$ which do not correspond to edges for $G$. Since $G$ is C*-nuclear, there exists a matrix $R_1 \in M_d(\cA)$ which is a positive completion of $R$. Let $P := R_1 + X Q_{\epsilon}^{-1} X^*$. Then $P + \epsilon I_d \in M_d(\cA)^+$, and $P + \epsilon I_d \geq X Q_{\epsilon}^{-1} X^*$, and $P$ is a positive completion of $P_{?} + \epsilon I_d$. By Lemma \ref{lem: matrix completion and schur} again,
    \[ \begin{pmatrix} P + \epsilon I_d & X Q_{\epsilon}^{-1/2} \\ Q_{\epsilon}^{-1/2} X^* & I_n \end{pmatrix} \geq 0. \]
    Conjugating by $I_d \oplus Q_{\epsilon}^{1/2}$, we obtain a positive completion
    \[ \begin{pmatrix} P & X \\ X^* & Q \end{pmatrix} + \epsilon I_{d+n} \in M_{d+n}(\cA)^+ \]
    and therefore $G+K_n$ is C*-nuclear.
\end{proof}





\begin{proposition} \label{prop: clique sum C*-nuclear}
    Suppose that $G$ and $H$ are C*-nuclear. Then, for every $n\in \mathbb{N}$ the $n$-clique sum $G$ and $H$ is C*-nuclear.
\end{proposition}


\begin{proof}
    We must show that the partially defined matrix
    \[ 
    \begin{pmatrix} P_? & X_? & Z_? \\ X_?^* & Q & Y_? \\ Z_?^* & Y_?^* & R_? \end{pmatrix} + \epsilon I_{d_1+n+d_2} \in M_{d_1+n+d_2}(\cA) 
    \] 
    has a positive matrix completion in $M_{d_1+n+d_2}(\cA)$ whenever it has a $B(\mathcal{H})$-valued positive completion, where $P_?, X_?, Y_?$ and $R_?$ are partially defined and $Z_?$ is completely unspecified. Since $G$ and $H$ are C*-nuclear, there exist $\cA$-valued completions $P, X, Y$ and $R$ of the partially-defined matrices such that
    \[ 
    \begin{pmatrix} P & X \\ X^* & Q \end{pmatrix} + \epsilon I_{d_1+n} \in M_{d_1+n}(\cA)^+ \quad \text{and} \quad \begin{pmatrix} Q & Y \\ Y^* & R \end{pmatrix} + \epsilon I_{n+d_2} \in M_{n+d_2}(\cA)^+. \] 
    By Lemma \ref{lem: matrix completion and schur}, there exists a matrix completion for $Z$ as desired. So the $n$-clique sum of $G$ and $H$ is C*-nuclear.
\end{proof}

We obtain the following clique blowup permanence of C*-nuclearity.

\begin{lemma} \label{lem:blow-up}
Let $H^{(m)}$ denote the $m$-clique blowup of a graph $H$. If $H$ is C*-nuclear, then $H^{(m)}$ is C*-nuclear.
\end{lemma}

\begin{proof}
After ordering the vertices by according to their fiber blowups, we have $\mathcal{S}_{H^{(m)}}\cong M_m(\mathcal{S}_H)$. Thus, for any C*-algebra $\mathcal{A}$, we have up to a reblocking map (and its inverse at the end) that via natural complete order isomorphisms,
$$
\mathcal{S}_{H^{(m)}} \otimes_c \mathcal{A} \cong M_m(\mathcal{S}_H \otimes_c \mathcal{A}) \cong M_m(\mathcal{S}_H \otimes_{min} \mathcal{A}) \cong \mathcal{S}_{H^{(m)}} \otimes_{min} \mathcal{A},
$$
so that $\mathcal{S}_{H^{(m)}}$ is C*-nuclear.
\end{proof}

\begin{theorem} \label{thm:ind-min}
Let $H$ be an induced minor of a finite graph $G$. If $G$ is C*-nuclear, then $H$ is C*-nuclear.
\end{theorem}

\begin{proof}
It is enough to prove permanence under vertex deletion and edge contraction. Let $\mathcal{A}$ be a unital C*-algebra.

We start with vertex deletion. Let $H=G \setminus \{v\}$. Let $P$ be an $\mathcal{A}$-valued $H$-partial matrix arising from a UCP map $\mathcal{S}_H\to \mathcal{A}$, and suppose that $P$ has a positive completion in $B(\mathcal{K})$. Enlarge $P$ to a $G$-partial matrix $\widetilde P$ by adjoining a zero row and column at $v$. Since the new diagonal entry is $0$, the sum of the diagonal entries remains $1_{\mathcal{A}}$, and a positive completion of $P$ enlarges to a positive completion of $\widetilde P$ by adjoining the same zero row and column. By $(\min,c)$-nuclearity of $\mathcal{S}_G$, for every $\varepsilon>0$ there is an $A$-valued completion $Q$ of $\widetilde P$ such that $Q+\varepsilon I\ge 0$. Taking the principal submatrix obtained by deleting the row and column of $v$ gives an $\mathcal{A}$-valued completion of $P$ which is positive after adding $\varepsilon I$. Hence $\mathcal{S}_H$ is $(\min,c)$-nuclear.

We now prove permanence under edge contraction. Let $e = \{0,1\} \in E(G)$, and let $H=G/e$. Denote the contracted vertex by $v$, and identify the remaining vertices of $H$ with $2,\dots,n-1$. Let $P=(p_{ij})$ be an $\mathcal{A}$-valued $H$-partial matrix arising from a ucp map $\mathcal{S}_H\to \mathcal{A}$. Assume that $P$ has a positive completion in $B(\mathcal{K})$. We must show that $P$ admits $\mathcal{A}$-valued completions which are positive up to an arbitrarily small scalar perturbation.

Set $t=\frac{\sqrt{2}}{2}$. We lift $P$ to a $G$-partial matrix $\widetilde P=(\widetilde p_{ij})$ by splitting the contracted vertex equally between the vertices $0$ and $1$. On the $\{0,1\}$ corner, we set $\widetilde p_{ij} = \frac{1}{2}p_{vv}$ for $i=0,1$ or $j=0,1$. For $j=0,1$, if $\{0,j\} \in E(G)$, we set $\widetilde p_{0j}=t\,p_{vj}$, and if $\{1,j\} \in E(G)$, we set $\widetilde p_{1j}=t\,p_{vj}$. For $i=0,1$, the entries $\widetilde p_{i0}$ and $\widetilde p_{i1}$ are defined as the adjoints of $\widetilde p_{0i}$ and $\widetilde p_{1i}$ respectively, and all other entries indexed by $\{2,\dots,n\}$ are inherited from $P$.

The diagonal sum of $\widetilde P$ is the same as that of $P$, since $\widetilde p_{00}+\widetilde p_{11}=p_{vv}$. Thus $\widetilde P$ is again the partial Choi matrix of a ucp map on $\mathcal{S}_G$. Now, let $T$ be the matrix corresponding to sending the basis vector $e_v$ to $t(e_0+e_1)$ (for $t=\frac{\sqrt{2}}{2}$) and fixing the remaining basis vectors. If $C$ is a positive completion of $P$, then the lifted matrix $TCT^*$ is a positive completion of $\widetilde P$.

Since $\mathcal{S}_G$ is $(\min,c)$-nuclear, for every $\varepsilon>0$ there is an $\mathcal{A}$-valued completion $Q=(q_{ij})$ of $\widetilde P$ such that $Q+\varepsilon I\ge 0$. Because the edge $\{0,1\}$ is specified, we have $q_{ij}=\frac12 p_{vv}$ for $i=0,1$ or $j=0,1$. We now show that the first two rows of $Q$ are asymptotically equal as $\varepsilon \rightarrow 0$.

Fix $j\geq 2$. Compress $Q+\varepsilon I$ to the two-dimensional subspace generated by $e_0-e_1$ and $e_j$. We obtain
$$
\begin{pmatrix}
2\varepsilon 1_{A} & q_{0j}-q_{1j}\\
q_{j0}-q_{j1} & p_{jj}+\varepsilon 1_{A}
\end{pmatrix}\ge 0.
$$
The standard $2\times2$ positivity estimate yields $(q_{j0}-q_{j1})(q_{0j}-q_{1j}) \leq 2\varepsilon\bigl(p_{jj}+\varepsilon1_{A}\bigr)$, so that 
$$
\|q_{0j}-q_{1j}\| \leq \sqrt{2\varepsilon\bigl(\|p_{jj}\|+\varepsilon\bigr)}.
$$ 

We now delete row and column $1$ from $Q$, and rescale row and column $0$ by $\sqrt2$. More precisely, let $Q_0$ be the compression of $Q$ to the index set $\{0,2,\dots,n-1\}$, $D=\operatorname{diag}(\sqrt2,1,\dots,1)$ and $R=DQ_0D$. Since $Q_0+\varepsilon I\ge 0$, we have $R+\varepsilon D^2\ge 0$, and therefore $R+2\varepsilon I\ge 0$. We identify the first coordinate of $R$ with the contracted vertex $v$.

The diagonal entry at $v$ is $R_{vv}=2q_{00}=p_{vv}$. Fix $j\geq 2$. If $\{v,j\} \in E(H)$ because $\{0,j\}\in E(G)$, then
$$
R_{vj}=\sqrt2\,q_{0j}
=\sqrt2\left(\frac{\sqrt{2}}{2}p_{vj}\right)
=p_{vj}.
$$
If instead $\{v,j\}\in E(H)$ only because $\{1,j\}\in E(G)$, then $q_{1j}$ is specified and $q_{1j}=\frac{\sqrt{2}}{2} p_{vj}$. Thus, 
$$
\begin{aligned}
\|R_{vj}-p_{vj}\|
=\|\sqrt2\,q_{0j}-p_{vj}\|
=\sqrt2\,\|q_{0j}-q_{1j}\|
\le 2\sqrt{\varepsilon\bigl(\|p_{jj}\|+\varepsilon\bigr)},
\end{aligned}
$$
and every specified entry of $P$ is the same as the one in $R$, except perhaps for finitely many entries of the form above, where the errors in those entries tend uniformly to zero with $\varepsilon$.

Change in $P$ those entries exactly, together with their corresponding adjoints. Let $E_\varepsilon$ denote the resulting self-adjoint correction, so that $R'=R+E_\varepsilon$ is a genuine $\mathcal{A}$-valued completion of $P$. Since the number of corrected entries is finite and every corrected entry tends to zero in norm, we get $\|E_\varepsilon\|\longrightarrow 0$ as $\varepsilon\rightarrow 0$. Hence, $R'+\bigl(2\varepsilon+\|E_\varepsilon\|\bigr)I\ge 0$. Now, since $2\varepsilon+\|E_\varepsilon\|\longrightarrow 0$, we obtain $\mathcal{A}$-valued completions of $P$ which are positive up to an arbitrarily small scalar perturbation. Hence $\mathcal{S}_H$ is $(\min,c)$-nuclear.

This proves permanence under edge contraction. The rest follows by induction, and therefore we obtain permanence under induced minors.
\end{proof}

We conclude with a general permanence property for operator systems, which will be needed later in the context of graph operator systems. Let $\mathcal{S}$ and $\mathcal{T}$ be operator systems. Then $\mathcal{S}$ is a \textit{retract} of $\mathcal{T}$ (in the sense of \cite{Harris}) if there exist UCP maps $\varphi: \mathcal{S} \to \mathcal{T}$ and $\psi: \mathcal{T} \to \mathcal{S}$ such that $\psi \circ \varphi$ is the identity map on $\mathcal{S}$.

\begin{proposition} \label{prop: retracts}
    Suppose that $\mathcal{S}$ is a retract of $\mathcal{T}$ and that $\mathcal{T}$ is C*-nuclear. Then $\mathcal{S}$ is C*-nuclear.
\end{proposition}

\begin{proof}
    It suffices to show that for any C*-algebra $\mathcal{A}$ the identity map $\mathcal{S} \otimes_{min} \mathcal{A} \to \mathcal{S} \otimes_{max} \mathcal{A}$ is UCP. Since the min and max tensor products are functorial \cite{KPTT1}, we have UCP maps $\varphi \otimes id: \mathcal{S} \otimes_{min} \mathcal{A} \to \mathcal{T} \otimes_{min} \mathcal{A}$ and $\psi \otimes id: \mathcal{T} \otimes_{max} \mathcal{A} \to \mathcal{S} \otimes_{max} \mathcal{A}$, where $\psi \circ \varphi$ is the identity map on $\mathcal{S}$. If $\mathcal{T}$ is C*-nuclear, then the identity map $\iota_{\mathcal{A}}$ from $\mathcal{T} \otimes_{min} \mathcal{A}$ to $\mathcal{T} \otimes_{max} \mathcal{A}$ is UCP, and it follows that the identity map $\mathcal{S} \otimes_{min} \mathcal{A} \to \mathcal{S} \otimes_{max} \mathcal{A}$ factors as $(\psi \otimes id) \circ \iota_{\mathcal{A}} \circ (\varphi \otimes id)$, and is therefore also UCP. Thus, $\mathcal{S}$ is C*-nuclear.
\end{proof}

\section{C*-nuclearity for 2-chordal graphs} \label{sec: 2-chordal implies C*-nuclear}

In this section, we prove that every 2-chordal graph is C*-nuclear. We will first prove that the 4-cycle $C_4$ is C*-nuclear. Then, by Part 3 of Theorem \ref{thm: McCullough} and the permanence properties of Section \ref{sec: Permanence}, we will conclude that every 2-chordal graph is C*-nuclear.

\begin{theorem} \label{C_4 C*-nuc}
Let $C_4$ be the cycle graph on four vertices, and let
$$
\mathcal{S}_{C_4} = 
\operatorname{span}
\{
E_{ii}:0\leq i\leq 3
\}
+
\operatorname{span}
\{
E_{01},E_{10},E_{12},E_{21},
E_{23},E_{32},E_{30},E_{03}
\}
\subseteq M_4.
$$
Then $\mathcal{S}_{C_4}$ is C*-nuclear.
\end{theorem}

\begin{proof}
Let $\mathcal{A}$ be a unital $C^*$-algebra, let $\varphi:\mathcal{S}_{C_4}\longrightarrow \mathcal{A}$ be a unital completely positive map, and let $\varepsilon>0$. It suffices to show
that the partially defined matrix $\big(\varphi(E_{ij})\big)_{(i,j)\in E(C_4)}$ admits a completion $B\in M_4(\mathcal{A})$ such that $B+\varepsilon I_4\otimes 1_{\mathcal A}\geq 0$.

For every diagonal or edge entry, write $a_{ij}=\varphi(E_{ij})$, and let
$$
P=
\begin{pmatrix}
a_{00} & a_{01} & 0      & a_{03} \\
a_{10} & a_{11} & a_{12} & 0 \\
0      & a_{21} & a_{22} & a_{23} \\
a_{30} & 0      & a_{32} & a_{33}
\end{pmatrix}
\in M_4(\mathcal A)_{\mathrm{sa}}.
$$
Now let
$$
\mathcal J=
\Bigg \{
\begin{pmatrix}
0 & 0 & x & 0 \\
0 & 0 & 0 & y \\
x^* & 0 & 0 & 0 \\
0 & y^* & 0 & 0
\end{pmatrix}
:x,y\in\mathcal A
\Bigg \}
\subseteq M_4(\mathcal A)_{\mathrm{sa}}.
$$
We will show that
\begin{equation} \label{eq:1}
P\in
\overline{M_4(\mathcal A)^++\mathcal J}.
\tag{1}
\end{equation}

To this end, let $F$ be a positive linear functional on
$M_4(\mathcal A)$ such that $F(\mathcal J)=0$. By GNS construction, there exist a unital representation $\pi:\mathcal A\longrightarrow B(\mathcal{H})$ and vectors $\xi_0,\ldots,\xi_3\in \mathcal{H}$ such that
\begin{equation*} 
F\big([b_{ij}]\big)
=\sum_{i,j=0}^3
\left\langle
\pi(b_{ij})\xi_j,\xi_i
\right\rangle
\ \ \
\text{for every} \ \ \ [b_{ij}]\in M_4(\mathcal A).
\end{equation*}
Since $F$ annihilates $\mathcal J$, applying $F$ to the elements of
$\mathcal J$ corresponding to $x$ and $ix$ (with $y=0$) respectively, shows that $\langle \pi(x)\xi_2,\xi_0\rangle=0$ for $x\in\mathcal A$. Similarly, we also obtain $\langle \pi(y)\xi_3,\xi_1\rangle=0$ for $y\in\mathcal A$.

For $0\leq i\leq 3$, we denote $\mathcal{H}_i=\overline{\pi(\mathcal A)\xi_i}$. It follows from the previous two equations that $\mathcal{H}_0\perp \mathcal{H}_2$ and $\mathcal{H}_1\perp \mathcal{H}_3$. Let $p_2$ and $p_3$ denote the orthogonal projections onto $\mathcal{H}_2$ and
$\mathcal{H}_3$, respectively. Since the spaces $\mathcal{H}_i$ reduce $\pi(\mathcal A)$,
we have $p_2,p_3\in\pi(\mathcal A)'$. Define the self-adjoint unitaries $u=1-2p_2$ and $v=1-2p_3$, so that $u\xi_2=-\xi_2$ and $u\xi_0=\xi_0$, while $v\xi_3=-\xi_3$ and $v\xi_1=\xi_1$.

We next construct a positive completion of the matrix
$[\pi(a_{ij})]$ which is invariant under both $u$ and $v$.
Since $\pi\circ\varphi:\mathcal S_{C_4}\to B(\mathcal{H})$ is UCP, Arveson's
extension theorem provides a UCP extension $\Psi:M_4\longrightarrow B(\mathcal{H})$ of $\pi \circ \varphi$. Its Choi matrix $C_\Psi=[\Psi(E_{ij})]_{i,j=1}^4$ is positive and is a completion of the partially defined matrix $[\pi(a_{ij})]$ with entries in $B(\mathcal{H})$.

Let $\mathcal K$ denote the set of all positive completions of this
partial matrix in $M_4(B(\mathcal{H}))$. Thus, we have that the set
\begin{equation*}
\mathcal K
=\{
[c_{ij}]\in M_4(B(\mathcal{H}))^+:
c_{ij}=\pi(a_{ij})
\text{ whenever }i=j
\text{ or }{(i,j)}\in E(C_4)
\},
\end{equation*}
is nonempty, convex, and compact in the entrywise
ultraweak topology. Indeed, positivity and the fact that the diagonal
entries are fixed give a uniform bound on each of the two unspecified
entries by taking either the $\{0,2\}$ positive corner, or the $\{1,3\}$ positive corner, and applying \cite[Exercise 3.2(iii)]{Pnbook} and the reverse triangle inequality.

Let $\Gamma=\langle u,v\rangle$ be the group that acts affinely on $\mathcal K$ by simultaneous conjugation, given by $g\cdot[c_{ij}]
=[gc_{ij}g^*]$ for $g\in\Gamma$. This action preserves the specified entries since $u,v\in\pi(\mathcal A)'$. Moreover, $\Gamma$ is generated by two involutions, and is therefore a
quotient of $\mathbb Z_2*\mathbb Z_2$, which is the infinite dihedral group. In particular, $\Gamma$ is amenable.

Applying the fixed-point property of amenable groups acting affinely on compact convex sets, we obtain an element $C=[c_{ij}]\in\mathcal K$ which is invariant under both $u$ and $v$.

The invariance under $u$, together with how $u$ fixes $\xi_0$ and negates $\xi_2$ yields
$$
\begin{aligned}
\left\langle c_{02}\xi_2,\xi_0\right\rangle
=\left\langle uc_{02}u\xi_2,\xi_0\right\rangle\
=\left\langle c_{02}u\xi_2,u\xi_0\right\rangle\
=-\left\langle c_{02}\xi_2,\xi_0\right\rangle.
\end{aligned}
$$
Therefore, $\langle c_{02}\xi_2,\xi_0\rangle=0$. Likewise, invariance under $v$ together with how $v$ fixes $\xi_1$ and negates $\xi_3$ yields $\langle c_{13}\xi_3,\xi_1\rangle=0$. The entries of $C$ corresponding to edges $(i,j) \in E(C_4)$ agree with those of $\pi(P)$, while the contributions of the two missing chords vanish since $\langle c_{02}\xi_2,\xi_0\rangle=0$ and $\langle c_{13}\xi_3,\xi_1\rangle=0$.
Consequently, since $C\geq0$, we therefore have, we get that  
$$
F(P) = \left\langle
C
\begin{pmatrix}
\xi_0\ \xi_1\ \xi_2\ \xi_3
\end{pmatrix},
\begin{pmatrix}
\xi_0\ \xi_1\ \xi_2\ \xi_3
\end{pmatrix}
\right\rangle \geq 0.
$$

We have shown that every positive linear functional on
$M_4(\mathcal A)$ which annihilates $\mathcal J$ takes a nonnegative
value on $P$. By the Hahn--Banach separation theorem, this is precisely the assertion that proves equation \eqref{eq:1}, namely, that $P\in\overline{M_4(\mathcal A)^+ +\mathcal J}$.

We may therefore choose $Q\in M_4(\mathcal A)^+$ and $Z\in\mathcal J$
such that $|P-Q-Z|<\varepsilon$. Define $R=P-Q-Z$. Since $R=R^*$ and $|R|<\varepsilon$, we get that $R+\varepsilon I_4\otimes1_{\mathcal A}\geq 0$, and it follows that
\begin{equation*}
P-Z+\varepsilon I_4\otimes1_{\mathcal A}=Q+\big(R+\varepsilon I_4\otimes1_{\mathcal A}\big)
\geq0.
\end{equation*}
Since $-Z\in\mathcal J$, the matrix $P-Z$ is a completion of the
original partial matrix $\big(\varphi(E_{ij})\big)_{(i,j)\in E(C_4)}$, and so that the partial matrix associated to $\varphi$ admits a completion which is positive after adding $\varepsilon I_4\otimes1_{\mathcal A}$. As $\mathcal A$, $\varphi$, and $\varepsilon$ were arbitrary, the positive-completion characterization in \cite[Theorem 3.10]{AKRSx} now implies that $\mathcal S_{C_4}$ is $(\min,c)$-nuclear.
\end{proof}

The next theorem is the first direction of our main result.

\begin{theorem} \label{thm: 2-chordal implies C*-nuclear}
    Suppose that $G$ is 2-chordal. Then the operator system $\cc S_G$ is C*-nuclear.
\end{theorem}

\begin{proof}
    From above we see that $C_4$ is C*-nuclear. By Proposition \ref{prop: G+K_n C*-nuclear}, the 4-wheel $W_4 = C_4 + K_1$ is C*-nuclear. By Lemma \ref{lem:blow-up} we get that $W_4^{(m)}$ is C*-nuclear for every $m \in \mathbb{N}$. It follows that every induced subgraph of $W_4^{(m)}$ is C*-nuclear by Theorem \ref{thm:ind-min}. By Proposition \ref{prop: clique sum C*-nuclear}, clique sums of these induced subgraphs are also C*-nuclear. It follows from Theorem \ref{thm: McCullough} that every 2-chordal graph is C*-nuclear.
\end{proof}

\section{Failure of C*-nuclearity for $K_{2,3}$ and $2K_1 + P_4$} \label{sec: K_2,3 and 2K_1+P_4 not C*-nuc}

In this section, we provide the first known examples of graph operator systems $\mathcal{S}_G$ which fail to be C*-nuclear. Specifically, we will give direct proofs that the forbidden subgraphs $K_{2,3}$ and $2K_1 + P_4$ are not C*-nuclear. We will consider the remaining forbidden subgraphs in the next section.

For both of these graphs, the proof will rely on a result of Kavruk from \cite{Kavrukw}. In the following, we write $x < y$ for $x,y \in \mathcal{A}$ and $\mathcal{A}$ a C*-algebra to mean $\delta I_{\mathcal{A}} \leq y-x$ for some $\delta > 0$. A C*-algebra $\mathcal{A} \subseteq B(\mathcal{H})$ is said to have the \textit{complete (2,3)-tight Riesz interpolation property} if for every $p \in \mathbb{N}$, whenever there exist $x_1, x_2, y_1, y_2, y_3 \in M_p(\mathcal{A})$ and $b \in M_p(B(\mathcal{H}))$ such that $x_i < b < y_j$ for all indices $i$ and $j$, then there exists $a \in M_p(\mathcal{A})$ such that $x_i < a < y_j$ for all indices $i$ and $j$. Kavruk proves in \cite[Theorem 7.4]{Kavrukw} that a C*-algebra has the complete (2,3)-tight Riesz interpolation property if and only if $\mathcal{A}$ has the \textit{weak expectation property} (or WEP).

If $\mathcal{N}$ is a von Neumann algebra, then it is well-known that  $\mathcal{N}$ has the WEP if and only if $\mathcal{N}$ is injective. It follows that if $\mathcal{N} \subseteq B(\mathcal{H})$ is not injective, then there exists $p \in \mathbb{N}$, $b \in M_p(B(\mathcal{H}))$ and $x_1, x_2, y_1, y_2, y_3 \in M_p(\mathcal{N})$ such that $x_i < b < y_j$ for all indices $i$ and $j$, but $x_i < a < y_j$ is false for every $a \in M_p(\mathcal{N})$ (for some indices $i$ and $j$). By slightly shrinking the intervals, we may moreover arrange that there is no self-adjoint $a \in M_p(\mathcal{N})$ satisfying
$$
x_1,x_2\leq a\leq y_1,y_2,y_3.
$$
Indeed, choose $\delta>0$ such that
$$
x_i+\delta 1\leq b\leq y_j-\delta 1
$$
for every $i=1,2$ and $j=1,2,3$. Replacing $x_i$ by $x_i+\frac{\delta}{2}1$ and $y_j$ by $y_j-\frac{\delta}{2}1$, a non-strict interpolant for the new data would give a strict interpolant for the original data.

We begin by showing that the graph operator system associated to $K_{2,3}$ is not C*-nuclear. 

\begin{theorem} \label{thm: K_23 not C*-nuc}
The graph $K_{2,3}$ is not C*-nuclear.
\end{theorem}

\begin{proof}
We label the bipartition of $K_{2,3} = 2K_1 + 3K_1$ as $\{v,w\}\sqcup\{1,2,3\}$. Let $\mathcal M\subseteq B(\mathcal{H})$ be a noninjective von Neumann algebra. there exists $p \in \mathbb{N}$, $b \in M_p(B(\mathcal{H}))$ and $x_1, x_2, y_1, y_2, y_3 \in M_p(\mathcal{M})$ such that $x_i \leq b \leq y_j$ for all indices $i$ and $j$, but $x_i \leq a \leq y_j$ is false for every $a \in M_p(\mathcal{M})$ (for some indices $i$ and $j$).

Denote $\mathcal N=M_p(\mathcal M)\subseteq B(\mathcal{K})$, where $\mathcal{K} = \mathcal{H} \otimes\mathbb C^p$. Set $x_3 := x_2$ and define $D_j=y_j-x_j$ and $C_j=y_j+x_j$. Thus $x_j \leq b \leq y_j$ for $j=1,2,3$. Choose $R>0$ sufficiently large so that $R>2\max\bigl\{\|b\|,\|x_j\|,\|y_j\|:j=1,2,3\bigr\}$. For $j=1,2,3$, consider
$$
F_j=
\begin{pmatrix}
R1_{\mathcal N}-D_j&C_j\\
C_j&R1_{\mathcal N}-D_j
\end{pmatrix}.
$$
Conjugating by the scalar Hadamard unitary shows that $F_j$ is unitarily equivalent to
$$
\begin{pmatrix}
R1_{\mathcal N}+2x_j&0\\
0&R1_{\mathcal N}-2y_j
\end{pmatrix},
$$
and hence $F_j > 0$ by the choice of $R$. Now, by taking the two block rows $P_j$ and $Q_j$ of $F_j^{1/2}$ we get
$$
F_j=
\begin{pmatrix}
P_j\\
Q_j
\end{pmatrix}
\begin{pmatrix}
P_j\\
Q_j
\end{pmatrix}^{*},
\ \ \text{for} \ \
P_j,Q_j\in M_{1,2}(\mathcal N).
$$
We then have $P_jP_j^*=Q_jQ_j^*=R1_{\mathcal N}-D_j$ and $P_jQ_j^*=C_j$. Now define
$$
\widehat P_j=
\begin{pmatrix}
P_j\\
0
\end{pmatrix},
\qquad
\widehat Q_j=
\begin{pmatrix}
Q_j\\
0
\end{pmatrix}
\in M_2(\mathcal N),
$$
where the zero denotes a $1\times2$ block row.

We denote by $1$ the unit of $M_2(\mathcal N)$. Consider the $K_{2,3}$-partial matrix with entries in $M_2(\mathcal N)$:
$$
\mathcal P=
\begin{pmatrix}
R1 &?&\widehat P_1&\widehat P_2&\widehat P_3\\
?&R1 &\widehat Q_1&\widehat Q_2&\widehat Q_3\\
\widehat P_1^*&\widehat Q_1^*&1 &?&?\\
\widehat P_2^*&\widehat Q_2^*&?&1 &?\\
\widehat P_3^*&\widehat Q_3^*&?&?&1 
\end{pmatrix}.
$$
We claim that $\mathcal P$ admits a positive completion in $B(\mathcal{K} \oplus \mathcal{K})$. To this end, define
$$
U_0=
\begin{pmatrix}
2b&0\\
0&0
\end{pmatrix}
\in M_2(B(\mathcal{K}))=B(\mathcal{K}\oplus \mathcal{K}), \ \ \ \text{and} \ \ \  A_0=
\begin{pmatrix}
R1 & U_0\\
U_0& R1
\end{pmatrix}.
$$
Since $\|U_0\|=2\|b\|<R$, we have $A_0>0$.

For $j=1,2,3$, denote the $4 \times 1$ matrix $Z_j = \begin{pmatrix} \widehat P_j & \widehat Q_j \end{pmatrix}^{\mathrm{T}}$ with entries in $M_2(\mathcal N)$. Using the identities above and merely permuting the block coordinates, we obtain
$$
A_0-Z_jZ_j^*
\cong
\begin{pmatrix}
D_j&2b-C_j\\
2b-C_j&D_j
\end{pmatrix}
\oplus R1_{\mathcal N}\oplus R1_{\mathcal N}.
$$
The first summand is unitarily equivalent, again by the Hadamard unitary, to the matrix
$$
\begin{pmatrix}
2(b-x_j)&0\\
0&2(y_j-b)
\end{pmatrix},
$$
which is strictly positive. Hence $A_0-Z_jZ_j^*>0$. Since $A_0>0$, the Schur complement criterion gives $\rho_j=1 -Z_j^*A_0^{-1}Z_j>0$. Now set $Z=\begin{pmatrix}Z_1&Z_2&Z_3\end{pmatrix}$ and $B_0=Z^*A_0^{-1}Z+\operatorname{diag}(\rho_1,\rho_2,\rho_3)$. By construction, every diagonal block of $B_0$ is the identity element $1 \in M_2(\mathcal N)$, and $B_0-Z^*A_0^{-1}Z = \operatorname{diag}(\rho_1,\rho_2,\rho_3)>0$. Therefore,
$$
\begin{pmatrix}
A_0&Z\\
Z^*&B_0
\end{pmatrix}
\geq0.
$$
This is a positive completion of $\mathcal P$ in $M_5(M_2(B(\mathcal{K})))$. 

Suppose, toward a contradiction, that $\mathcal S_{K_{2,3}}$ is $(\min,c)$-nuclear. By \cite[Theorem 3.10]{AKRSx} and Proposition \ref{prop: psd completion in A''}, and since $M_2(\mathcal N) = M_2(\mathcal N)''$, we get a positive completion of $\mathcal P$ whose entries all belong to $M_2(\mathcal N)$. Let $U\in M_2(\mathcal N)$ denote the $(v,w)$-entry of this completion. For each $j=1,2,3$, positivity of the principal submatrix corresponding to the vertices $\{v,w,j\}$ yields
$$
\begin{pmatrix}
R1 &U&\widehat P_j\\
U^*& R1 &\widehat Q_j\\
\widehat P_j^*&\widehat Q_j^*& 1 
\end{pmatrix}
\geq0.
$$
Taking the Schur complement of the lower-right identity block gives
$$
\begin{pmatrix}
R1 -\widehat P_j\widehat P_j^*
&
U-\widehat P_j\widehat Q_j^*
\\
U^*-\widehat Q_j\widehat P_j^*
&
R1 -\widehat Q_j\widehat Q_j^*
\end{pmatrix}
\geq0.
$$

Let $e=
\begin{pmatrix}
1_{\mathcal N}&0\\
0&0
\end{pmatrix}
\in M_2(\mathcal N)$, and identify $e M_2(\mathcal N) e$ with $\mathcal N$. Denote $u=eUe\in\mathcal N$. Compressing the preceding inequality by $\operatorname{diag}(e,e)$ yields $\begin{pmatrix}
D_j&u-C_j\\
u^*-C_j&D_j
\end{pmatrix}
\geq 0$.

Set $h=\frac{u+u^*}{2}\in\mathcal N_{\mathrm{sa}}$. Conjugating the preceding matrix by the flip unitary and averaging with the original matrix gives $\begin{pmatrix}
D_j&h-C_j\\
h-C_j&D_j
\end{pmatrix}
\geq 0$. A final conjugation by the Hadamard unitary yields $D_j+h-C_j\geq 0$ and $D_j-h+C_j\geq 0$. 
Setting $a=\frac h2\in\mathcal N_{\mathrm{sa}}$, and recalling that $D_j=y_j-x_j$ and $C_j=y_j+x_j$, we conclude that $x_j\leq a\leq y_j$ for $j=1,2,3$. By the choice of the pairs $(x_j,y_j)$, this gives
$$
x_1,x_2\leq a\leq y_1,y_2,y_3,
$$
contradicting the choice of our interpolation data. Therefore $\mathcal S_{K_{2,3}}$ is not $(\min,c)$-nuclear.
\end{proof}

\begin{theorem} \label{2K1+P4 not C*-nuc}
The graph $2K_1 + P_4$ is not C*-nuclear.
\end{theorem}


\begin{proof}
Denote the two nonadjacent vertices of $2K_1$ by $v,w$, with the $4$-path $P_4$ given by four vertices  $v_1-v_2-v_3-v_4$. Choose a noninjective von-Neumann algebra $\mathcal{N} \subseteq B(\mathcal{H})$. Passing to a matrix amplification if necessary, and as before, by shrinking the intervals, we may assume there exist
$$
x_1,x_2,y_1,y_2,y_3\in \mathcal{N}_{sa},
\ \ \text{and} \ \ b\in B(\mathcal{H})_{sa},
$$
such that
$$
x_1,x_2<b<y_1,y_2,y_3,
$$
while there is no $a\in \mathcal{N}_{sa}$ satisfying
$$
x_1,x_2\leq a\leq y_1,y_2,y_3.
$$
Set $x_3 := x_2$. For $i=1,2,3$, define $D_i=y_i- x_i$ and $C_i=y_i+ x_i$, and set
$$
G_i=
\begin{pmatrix}
-D_i&C_i\\
C_i&-D_i
\end{pmatrix}\in M_2(\mathcal{N})_{sa}.
$$
Choose $R>0$ sufficiently large that the four matrices
$$
W_1=\frac R2I_2+G_1,
\ \ 
W_2=\frac R2I_2,
\ \ 
W_3=\frac R2I_2+G_2,
\ \ 
W_4=\frac R2I_2+G_3-G_2
$$
are all positive. They must then satisfy $W_i+W_{i+1}=RI_2+G_i$ for $i=1,2,3$. Factor
$$
W_i=
\binom{P_i}{Q_i}
\begin{pmatrix}P_i^*&Q_i^*\end{pmatrix},
\ \ \text{with} \ \ P_i,Q_i\in M_{1,2}(\mathcal{N}).
$$
Embed these rows into $\mathcal{A}:=M_2(\mathcal{N})$ by
$$
\widehat P_i=
\begin{pmatrix}P_i\\0\end{pmatrix},
\ \ \text{and} \ \ 
\widehat Q_i=
\begin{pmatrix}Q_i\\0\end{pmatrix},
$$
where the second row is zero. Consider the partially-defined $\mathcal{A}$-valued matrix $N=(n_{i,j})$ defined by $n_{vv}=n_{ww}=R1_A, \ \ n_{v_iv_i}=1_A, \ \ n_{vv_i}=\widehat P_i, \ \ n_{wv_i}=\widehat Q_i$, for $i=1,2,3,4$, and $n_{v_iv_{i+1}}=0$ for $i=1,2,3$. All remaining unspecified entries are precisely the nonedges of $2K_1+P_4$.

Represent $M_2(\mathcal{N}) \subseteq B(\mathcal{H} \oplus \mathcal{H})$. We will now show that the partial matrix $N$ has a positive completion in $M_2(B(\mathcal{H})) \cong B(\mathcal{H} \oplus \mathcal{H})$. Fill the missing $vw$-entry by
$$
U_0=
\begin{pmatrix}
2b & 0 \\ 0 & 0
\end{pmatrix}\in B(\mathcal{H} \oplus \mathcal{H}).
$$
After this entry is filled, the specified graph is $K_2 + P_4$, which is chordal. Its maximal cliques are $\{v,w,v_i,v_{i+1}\}$ for $i=1,2,3$ . It is therefore enough, by the operator-valued chordal completion theorem, to check positivity on these three clique submatrices.

Fix $i=1,2,3$. Taking the Schur complement of the lower-right
$1_{\mathcal{A}} \oplus 1_{\mathcal{A}}$ block and using the equations preceding the lemma, the resulting $2\times2$ matrix over $M_2(\mathcal{N})$ is, after permuting the two copies of $\mathcal{H}$ inside $\mathcal{H} \oplus \mathcal{H}$, unitarily equivalent to
$$
\begin{pmatrix}
D_i&2b-C_i\\
2b-C_i&D_i
\end{pmatrix}
\oplus R1_{\mathcal{N}} \oplus R1_{\mathcal{N}}.
$$
Conjugating the first summand by the scalar Hadamard unitary gives
$$
\begin{pmatrix}
2(b-x_i)&0\\
0&2(y_i-b)
\end{pmatrix}>0.
$$
Hence every maximal clique submatrix is positive, and the chordal
completion theorem yields a positive completion.

Suppose, toward a contradiction, that $2K_1 + P_4$ is C*-nuclear. By \cite[Theorem 3.10]{AKRSx} and Proposition \ref{prop: psd completion in A''} it follows that $N$ has a $\mathcal{A}$-valued positive completion, since $\mathcal{A} = M_2(\mathcal{N}) = M_2(\mathcal{N})''$. Let $U\in M_2(\mathcal{N})$ be the $vw$-entry of this positive completion. For $i=1,2,3$, positivity of the principal submatrix on $\{v,w,v_i,v_{i+1}\}$ and the Schur complement of the two identity blocks give
\begin{equation} \label{eq:10}
\begin{pmatrix}
R1_A-\widehat P_i\widehat P_i^*
        -\widehat P_{i+1}\widehat P_{i+1}^*
&
U-\widehat P_i\widehat Q_i^*
 -\widehat P_{i+1}\widehat Q_{i+1}^*\\
U^*-\widehat Q_i\widehat P_i^*
 -\widehat Q_{i+1}\widehat P_{i+1}^*
&
R1_A-\widehat Q_i\widehat Q_i^*
        -\widehat Q_{i+1}\widehat Q_{i+1}^*
\end{pmatrix}\geq 0.
\end{equation}
Let $e=\begin{pmatrix}1_{\mathcal{N}} &0\\0&0\end{pmatrix}\in \mathcal{A}$ and $u=eUe\in \mathcal{N}$. Compressing \eqref{eq:10} by $\operatorname{diag}(e,e)$ and using the equations preceding the lemma, we get
$$
\begin{pmatrix}
D_i&u-C_i\\
u^*-C_i&D_i
\end{pmatrix}\geq 0,
\ \ \text{for} \ \ i=1,2,3.
$$
Conjugate this inequality by the flip unitary and average with it, so that if $h=\frac{u+u^*}{2}\in \mathcal{N}_{sa}$, we obtain
$$
\begin{pmatrix}
D_i&h-C_i\\h-C_i&D_i
\end{pmatrix}\geq 0.
$$
Hadamard conjugation of this inequality yields $D_i+h-C_i\geq 0$ and $D_i-h+C_i\geq 0$. Since $D_i=y_i-x_i$ and $C_i=y_i+x_i$, setting $a=\frac h2$ gives $x_i\leq a\leq y_i$ for $i=1,2,3$. Thus
$$
x_1,x_2\leq a\leq y_1,y_2,y_3,
$$
which is a contradiction.
\end{proof}

\section{Failure of C*-nuclearity for non-2-chordal graphs.} \label{sec: not 2-chordal implies not C*-nuclear}

We conclude the paper by addressing the remaining four forbidden subgraphs $C_5, \overline{C_6}, \overline{P_6}$ and $K_{2,2,2}$. We will first show that each of the graphs $C_6, \overline{C_6}$ and $\overline{P_6}$ fail to be C*-nuclear by using combinatorial arguments. Specifically, we will use the C*-nuclearity perserving operations from Section \ref{sec: Permanence} to transform these graphs into the graph $2K_1 +P _4$, which fails to be C*-nuclear by Theorem \ref{2K1+P4 not C*-nuc}. We consider the graph $K_{2,2,2}$ separately. In that case, we will give a direct proof that the corresponding operator system $\mathcal{S}_{K_{2,2,2}}$ cannot be C*-nuclear.

We begin with the graph $C_5$. Observe that since C*-nuclearity is perserved by both clique blow-ups and vertex deletions, we can combine these operations by replacing only some chosen vertices with cliques and leaving the other vertices unchanged. For a visualization of the next two proofs, see Figure \ref{fig: proof illustration}.

\begin{theorem} \label{thm: C_5 not C*-nuc}
The graph $C_5$ is not C*-nuclear.
\end{theorem}

\begin{proof}
Label the vertices of $C_5$ by $0,1,2,3,4$ cyclically. Replace each of the vertices in $\{1,2,3,4\}$ with a 2-clique and denote the resulting vertices by $i^+$ and $i^-$ for $i=1,2,3,4$. We show that $2K_1 + P_4$ is an induced minor of the resulting graph, so that by a combination of Lemma \ref{lem:blow-up} and Theorem \ref{thm:ind-min} we will have that $C_5$ does not have a C*-nuclear graph operator system.

Let $v,w$ be the two nonadjacent vertices of the $2K_1$ part and let
$v_1-v_2-v_3-v_4$ be its induced path. Take the partition of connected sets
$$
\begin{aligned}
B_v&=\{0\},&
B_w&=\{2^+,3^+\},\\
B_{v_1}&=\{1^-\},&
B_{v_2}&=\{1^+,2^-\},\\
B_{v_3}&=\{3^-,4^+\},&
B_{v_4}&=\{4^-\}.
\end{aligned}
$$
The connected sets $B_v$ and $B_w$ are disconnected from each other because vertex $0$ is nonadjacent in $C_5$ to both $2$ and $3$. The four connected sets $B_{v_1},\ldots,B_{v_4}$ form the $4$-path, so that consecutive sets are adjacent, while each of $B_v,B_w$ is adjacent to each $B_{v_i}$. Hence the indicated clique sets form an induced minor model of $2K_1 + P_4$.
\end{proof}

\begin{theorem} \label{thm: C6bar and P6bar not C*-nuc}
    The graphs $\overline{C_6}$ and $\overline{P_6}$ are not C*-nuclear.
\end{theorem}

\begin{proof}
    We first consider $\overline{C_6}$. Label the vertices of $C_6$ as $\{0,1,2,3,4,5\}$ and consider the graph complement. We will transform $\overline{C_6}$ into $2K_1 + P_4$ by a sequence of clique blowups and edge contractions. First we blow up the vertices 4 and 5 by replacing them with 2-cliques. Label the new vertices $\{ 0, 1, 2, 3^+, 3^-, 4^+, 4^-, 5\}$, and take the partition of connected sets
$$
\begin{aligned}
    B_{v}&=\{0\},&
    B_{w}&=\{1\},\\
    B_{v_1}&=\{3^-\},&
    B_{v_2}&=\{5,3^+\},\\
    B_{v_3}&=\{2,4^+\},&
    B_{v_4}&=\{4^-\}.
\end{aligned}
$$    
    The resulting subgraph formed by vertices $B_{v_1}, ... , B_{v_4}$ form the $4$-path, while the remaining vertices $B_v$ and $B_w$ are not connected to each other, but are connected to each of the other vertices. Thus the resulting graph is isomorphic to $2K_1 + P_4$. Since $2K_1 + P_4$ is obtained from $\overline{C_6}$ by operations preserving C*-nuclearity, but $2K_1 + P_4$ is not C*-nulear, we conclude that $\overline{C_6}$ is not C*-nuclear.

    Now we consider the graph $\overline{P_6}$. Again labeling the vertices $\{0,1,2,3,4,5\}$, the resulting graph is the same as $\overline{C_6}$ except it has one additional edge, say $(1,2)$. Taking the same clique blowup and the same partition of connected sets that we had for $\overline{C_6}$, we obtain the same graph $G$ whose vertices are $B_v, B_w, B_{v_1}, B_{v_2}, B_{v_3}, B_{v_4}$ as described in the previous paragraph. We conclude that $\overline{P_6}$ is not C*-nuclear.
\end{proof}

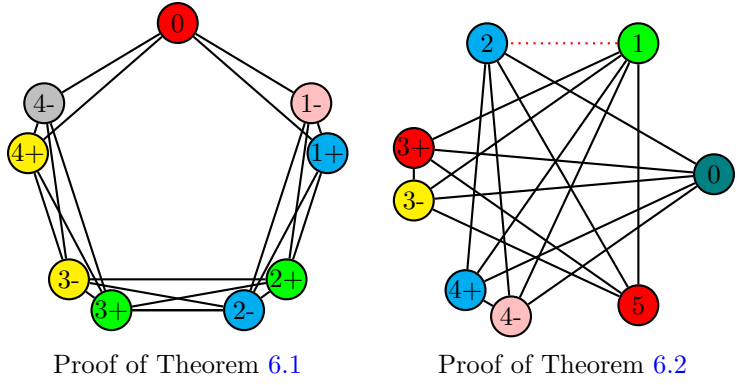
\begin{figure*}[h!]
    \centering
    \begin{subfigure}
    \centering
    {
    \begin{tikzpicture}[
      node distance=2cm,
      thick,
      main node/.style={circle, draw, fill=black, minimum size=15pt, inner sep=0pt},
      scale=1.0
    ]
    
      \node[main node, fill=red] (0) at (90:2) {0};
      \node[main node, fill=cyan] (1+) at (8:2) {1+};
      \node[main node, fill=pink] (1-) at (28:2) {1-};
      \node[main node, fill=green] (2+) at (316:2) {2+};
      \node[main node, fill=cyan] (2-) at (296:2) {2-};
      \node[main node, fill=green] (3+) at (244:2) {3+};
      \node[main node, fill=yellow] (3-) at (224:2) {3-};
      \node[main node, fill=yellow] (4+) at (172:2) {4+};
      \node[main node, fill=lightgray] (4-) at (152:2) {4-};
    
      \draw[black, thick] (1+) -- (2+);
      \draw[black, thick] (1+) -- (2-);
      \draw[black, thick] (1-) -- (2+);
      \draw[black, thick] (1-) -- (2-);
      \draw[black, thick] (1+) -- (1-);
      \draw[black, thick] (2+) -- (3+);
      \draw[black, thick] (2+) -- (3-);
      \draw[black, thick] (2-) -- (3+);
      \draw[black, thick] (2-) -- (3+);
      \draw[black, thick] (2-) -- (3-);
      \draw[black, thick] (2-) -- (2+);
      \draw[black, thick] (3+) -- (4+);
      \draw[black, thick] (3+) -- (4-);
      \draw[black, thick] (3-) -- (4+);
      \draw[black, thick] (3-) -- (4-);
      \draw[black, thick] (3-) -- (3+);
      \draw[black, thick] (4+) -- (0);
      \draw[black, thick] (4-) -- (0);
      \draw[black, thick] (4+) -- (4-);
      \draw[black, thick] (0) -- (1+);
      \draw[black, thick] (0) -- (1-);
    
      \node at (0,-2.5) {Proof of Theorem \ref{thm: C_5 not C*-nuc}};

    \end{tikzpicture}
    }
    \end{subfigure}%
    ~ 
    \begin{subfigure}
    \centering
    {
    \begin{tikzpicture}[
      node distance=2cm,
      thick,
      main node/.style={circle, draw, fill=black, minimum size=15pt, inner sep=0pt},
      scale=1.0
    ]
    
      \node[main node, fill=green] (1) at (60:2) {1};
      \node[main node, fill=cyan] (2) at (120:2) {2};
      \node[main node, fill=red] (3+) at (170:2) {3+};
      \node[main node, fill=yellow] (3-) at (190:2) {3-};
      \node[main node, fill=cyan] (4+) at (230:2) {4+};
      \node[main node, fill=pink] (4-) at (250:2) {4-};
      \node[main node, fill=red] (5) at (300:2) {5};
      \node[main node, fill=teal]
      (0) at (0:2) {0};
    
      \draw[black, thick] (3-) -- (1);
      \draw[black, thick] (3+) -- (1);
      \draw[black, thick] (1) -- (5);
      \draw[black, thick] (4+) -- (2);
      \draw[black, thick] (4-) -- (2);
      \draw[black, thick] (4+) -- (1);
      \draw[black, thick] (4-) -- (1);
      \draw[black, thick] (5) -- (2);
      \draw[black, thick] (5) -- (3+);
      \draw[black, thick] (5) -- (3-);
      \draw[black, thick] (0) -- (2);
      \draw[black, thick] (0) -- (3+);
      \draw[black, thick] (0) -- (3-);
      \draw[black, thick] (0) -- (4+);
      \draw[black, thick] (0) -- (4-);
      \draw[black, thick] (3+) -- (3-);
      \draw[black, thick] (4+) -- (4-);
      \draw[red, dotted] (1) -- (2);
    
      \node at (0,-2.5) {Proof of Theorem \ref{thm: C6bar and P6bar not C*-nuc}};

    \end{tikzpicture}
    }
    \end{subfigure}%
    \caption{Illustration of proofs} \label{fig: proof illustration}
\end{figure*}

It only remains to show that the graph $K_{2,2,2}$ is not C*-nuclear. Rather than solving a matrix completion problem or applying combinatorial techniques, we will show directly that the associated operator system is not C*-nuclear leveraging Kavruk's Namioka-Phelps test space from \cite{KavrukNP}.

We will show that the operator system $\mathcal{T} = \mathcal{S}_{K_{2,2,2}}$ is not C*-nuclear. It suffices to find a retract of $\mathcal{T}$ which is not C*-nuclear, and our candidate will be the operator system $\mathcal{W}_{2,3}$ introduced by Kavruk in \cite{KavrukNP}. This system is given by
\[ \mathcal{W}_{2,3} = \{ \text{diag}(a,b,c,x,y,z) \in M_6 : a+b+c = x+y+z \}. \]
In \cite[Theorem 0.3]{KavrukNP} Kavruk and shows that $\mathcal{W}_{2,3}$ is not C*-nuclear. 

\begin{theorem} \label{thm: K222 not C*-nuc}
    The operator system $\mathcal{W}_{2,3}$ is a retract of $\mathcal{S}_{K_{2,2,2}}$. Consequently $K_{2,2,2}$ is not C*-nuclear. 
\end{theorem}

\begin{proof}
    By a suitable labeling of the vertices of $K_{2,2,2}$, we may identify the operator system $\mathcal{T} :=\mathcal{S}_{K_{2,2,2}}$ with the operator system \[ \{ (A_{i,j})_{i,j=1}^6 \in M_6 : A_{i,j} = 0 \text{ whenever } |i-j|=3 \ \}. \]
    Let $P \in M_6$ denote the canonical shift operator on $\mathbb{C}^6$ (i.e. $P e_i = e_{i+1}$ for $i=1,2,3,4,5$ and $P e_6 = e_1$). This matrix is a unitary permutation matrix. Define $\rho(A) = \frac{1}{6} \sum_{n=0}^5 P^{-n} A P^n$. Direct calculation shows that the resulting matrix is the circulant $6 \times 6$ matrix
    \[ \rho(A) = c_0 I + c_1 P + c_2 P^2 + \dots + c_5 P^5 \] where $c_i$ is the average of the entries of $A$ on the $i$-th superdiagonal and the $(6-i)$-th subdiagonal (with $c_0$ being the average of the diagonal entries). Since the entries on the third superdiagonal and third subdiagonal of the elements of $\mathcal{T}$ are all zero, \[ \rho(T) = \{ c_0 I + c_1 P + c_2 P^2 + c_4 P^4 + c_5 P^5 : c_0,c_1, c_2, c_4,c_5 \in \mathbb{C} \}.\] Clearly $\rho(T)$ is a retract of $\mathcal{S}$, since it is a subsystem of $\mathcal{S}$ and since $\rho$ is UCP. 
    
    We conclude by showing that $\rho(T)$ is completely order isomorphic to $\mathcal{W}_{2,3}$. Let $F=(F_{i,j})$ denote the normalized discrete Fourier transform operator on $\mathbb{C}^6$, i.e. $F_{i,j}= \frac{1}{\sqrt{6}} \omega^{(6-i+1)(j-1)}$ where $\omega = e^{i\pi/3}$. Then \[ F^* P F = \text{diag}(1, \omega, \omega^2, \omega^3, \omega^4, \omega^5) \] and hence $F^* P^n F = \text{diag}(1, \omega^n, \omega^{2n}, \omega^{3n}, \omega^{4n}, \omega^{5n})$ for each $n \in \mathbb{N}$. It is easily verified that \[ \omega^0 + \omega^{2n} + \omega^{4n} = \omega^{n} + \omega^{3n} + \omega^{5n} \] for each $n \in \{0,1,2,4,5\}$. Since the set of matrices $\{F^*P^nF: n=0,1,2,4,5\}$ is linearly independent and since the vector space of $6 \times 6$ diagonal matrices is six-dimensional, it follows that \[ \text{span} \{F^* P^n F : n=0,1,2,4,5\} = \{ \text{diag}(a,x,b,y,c,z) : a+b+c = x+y+z\}. \] This operator system is obviously completely order isomorphic to $\mathcal{W}_{2,3}$. It follows that $\mathcal{W}_{2,3}$ is completely order isomorphic (in fact, unitarily equivalent) to the operator system $\rho(\mathcal{T})$. Since $\mathcal{W}_{2,3}$ is not C*-nuclear, neither is $\rho(T)$. By Proposition \ref{prop: retracts}, we conclude that $\mathcal{T}$ is not C*-nuclear, so the graph $K_{2,2,2}$ is not C*-nuclear.
\end{proof}

We have now shown that all of the forbidden subgraphs of Theorem \ref{thm: McCullough} fail to be C*-nuclear. This proves the other direction of our main result.

\begin{theorem} \label{thm: not 2-chordal implies not C*-nuclear}
    A graph $G$ is 2-chordal if and only if the operator system $\cc S_G$ is C*-nuclear.
\end{theorem}

\begin{proof}
    By Theorem \ref{thm: 2-chordal implies C*-nuclear}, it remains to show that if $G$ is not 2-chordal, then $\mathcal{S}_G$ is not C*-nuclear.

    Suppose that $G$ is not 2-chordal and, for the sake of contradiction, suppose that $G$ is C*-nuclear. Then every induced subgraph of $G$ is also C*-nuclear. By Theorem \ref{thm: McCullough}, $G$ contains one of the forbidden subgraphs depicted in Figure \ref{Fig: forbidden subgraphs}. Since induced subgraphs are induced minors, it follows from Theorem \ref{thm:ind-min} that one of these forbidden subgraphs is C*-nuclear. However, we have just shown above that each of the forbidden subgraphs fails to be C*-nuclear, so we have reached a contradiction. We conclude that $G$ is not C*-nuclear.
\end{proof}

\appendix

\section{On Riesz Interpolation, by Ali Kavruk}

For the convenience of the reader, we include this appendix on conjunction of WEP and Riesz interpolation properties. With this opportunity, we summarize some results on the topic. 
In \cite{LanceWEP}, Lance introduced the weak expectation property, which has a fundamental role in operator algebras, particularly in tensor and nuclearity theory. For von Neumann algebras it is equivalent to injectivity and for C*-algebras it describes a near-injectivity property by allowing conditional expectations into biduals. Since classical objects, such as Banach spaces or Kadison spaces, are exact, WEP-related properties coincide with categorical nuclearity. Therefore, WEP is a non-commutative algebra property. This was studied in various areas in operator algebras including operator systems, Hilbert C*-modules, real operator systems, etc. 

An operator system $\cl S$ is said to have WEP if for some embedding $i: \cl S \hookrightarrow B(\cl H)$ we have a ucp maps $\varphi: B(\cl H)\rightarrow \cl S^{**}$ with $\varphi\circ i =j$ where $j$ is the canonical inclusion $\cl S$ into bidual operator system $ \cl S^{**}$. 
$$
\xymatrix{
 \cl S \ar@{_{(}->}[d]^{i} \ar@{^{(}->}[rr]^{j} & & \cl S^{**}   \\
 B(\cl H)   \ar@{.>}[rru]_{ \varphi} }
$$
Since $B(\cl H)$ is injective, the definition is independent of a particular embedding of $\cl S$ into an injective object. Identifying $\cl S$ with its canonical image in $\cl S^{**}$, the reader may verify that WEP is equivalent to the following: for any operator systems $\cl T_1 \subseteq \cl T_2$, every ucp map $\varphi: \cl T_1 \rightarrow \cl S $ extends to a ucp map $\tilde \varphi :\cl T_2 \rightarrow \cl S^{**}$.


Riesz interpolation or decomposition theories are a common tool in classical objects such as function systems or ordered groups. 
Given operator systems $\cl S \subseteq \cl T$, we say that $\cl S$ has \textit{tight} $(n,k)$-\textit{Riesz interpolation property in} $\cl T$ if whenever the elements $x_1,...,x_n, y_1,..,y_k \in \cl S_{sa}$ interpolate in $\cl T$, meaning that there is an element $t\in \cl T_{sa}$ satisfying $x_i< t < y_j$ for $i=1,...,n$, $j=1,...,k$, then these elements interpolate in  $\cl S_{sa}$.  $\cl S$ is said to have \textit{complete tight} $(n,k)$-\textit{Riesz interpolation property in} $\cl T$ if an identical property holds for $M_m(\cl S) \subseteq M_m(\cl T)$ for every $m$. We remove the term $(n,k)$ if these properties are satisfied for every $n$ and $k$.

The main purpose of this appendix is to outline, for a fixed pair $(n,k) \gneq(2,2)$, the proof showing that the complete tight $(n,k)$-Riesz interpolation property implies WEP for C*-algebras, as presented in \cite{Kavrukw}. The machinery we use in this context is based on the tensor, quotients, and nuclearity theory in the operator system category initiated by Choi and Effros in \cite{CEinj, EffrosAspect} and systemically developed in \cite{Kavruk, KPTT1, KPTT2}. In the following we only highlight the main ideas from statement (\ref{AwithRI}) to statement (\ref{AwithWEP}) and we refer the reader to \cite{Kavrukw} for a detailed proof. The desired equivalence follows from a careful analysis of minimal and maximal cones generated by tensor products with non-commutative polygons. We fix a unital C*-subalgebra $\cl A$ of $B(\cl H)$ and consider
$$
\frac{\ell^\infty_{n}\oplus \ell^\infty_k} {\mathrm{span}\{\sigma_{n,k}\}}, \mbox{ where } \sigma_{n,k} = I_n\oplus -I_k.$$
Based on the projectivity of the maximal tensor product and properties of operator system quotients we obtain that the following statements are equivalent: 
\begin{equation}\label{AwithRI}
\cl A \mbox{ has the complete }(n,k)\mbox{-tight Riesz interpolation property in } B(\cl H),
\end{equation}
\begin{equation}\label{AwithB(H)} \mbox{we have a complete order embedding }\hspace{6.3cm}
\end{equation}
$$
\cl A \otimes_{max}\frac{\ell^\infty_{n}\oplus \ell^\infty_k} {\mathrm{span}\{\sigma_{n,k}\}}\subseteq B(\cl H) \otimes_{max} \frac{\ell^\infty_{n}\oplus \ell^\infty_k} {\mathrm{span}\{\sigma_{n,k}\}}.
$$
Since $\ell^\infty_{n}\oplus \ell^\infty_k/ \mathrm{span}\{\sigma_{n,k}\}$ has the lifting property and $B(\cl H)$ has WEP, an operator system version of a classical result of Kirchberg \cite{Kirchberg93, Kavruk} ensures that the minimal and the maximal tensor products coincide on 
$$
B(\cl H) \otimes \frac{\ell^\infty_{n}\oplus \ell^\infty_k} {\mathrm{span}\{\sigma_{n,k}\}}.
$$
Using the injectivity of the minimal tensor product, statement (\ref{AwithB(H)}) is equivalent to the complete order isomorphism of
\begin{equation}\label{AwithQuotient}
\cl A \otimes_{min} 
\frac{\ell^\infty_{n}\oplus \ell^\infty_k} {\mathrm{span}\{\sigma_{n,k}\}} 
= \cl A \otimes_{max} 
\frac{\ell^\infty_{n}\oplus \ell^\infty_k} {\mathrm{span}\{\sigma_{n,k}\}} .
\end{equation}
For the next step, we consider the full group C*-algebra $C^*(\bb Z_n * \bb Z_k)$ and use the identifications
$$
C^*(\bb Z_n * \bb Z_k) \cong C^*(\bb Z_n)*_1 C^*(\bb Z_k) \cong \ell^\infty_n *_1 \ell^\infty_k.
$$
We observe that the operator system quotient $\ell^\infty_{n}\oplus \ell^\infty_{k}/\mathrm{span}\{\sigma_{n,k}\}$ coincides with the coproduct
$
\ell^\infty_{n}+\ell^\infty_{k} \subseteq \ell^\infty_n *_1 \ell^\infty_k
$ completely order isomorphically. Mostly due to a result of Boca on the expansion of ucp maps on free products \cite{Boca91}, we deduce that statement (\ref{AwithQuotient}) is equivalent to 
\begin{equation} \label{AwithC*}
\cl A \otimes_{min} C^*(\bb Z_n * \bb Z_k) = 
\cl A \otimes_{max} C^*(\bb Z_n * \bb Z_k).
\end{equation}
For any case with $n,k \geq2$, $(n,k)\neq (2,2)$, the free group on countably infinite number of generators, $\bb F_\infty$, embeds into $\bb Z_n* \bb Z_k$ and, therefore, $C^*(\bb F_\infty)$ is a retract of $C^*(\bb Z_n * \bb Z_k)$ (we refer the reader to \cite{Pisier_intr} for an excellent survey). So, we deduce that statement (\ref{AwithC*}) implies statement (\ref{AwithC*F}).
\begin{equation} \label{AwithC*F}
\cl A \otimes_{min} C^*(\bb F_\infty) = 
\cl A \otimes_{max} C^*(\bb F_\infty).
\end{equation} 
\begin{equation} \label{AwithWEP}
\cl A \mbox{ has WEP}.
\end{equation} 
The equivalence of statements (\ref{AwithC*F}) and (\ref{AwithWEP}) is a well known WEP characterization of Kirchberg in his seminal work \cite{Kirchberg93}. On the other hand the implication from statement (\ref{AwithWEP}) to statement (\ref{AwithC*}) follows from the fact that $C^*(\bb Z_n * \bb Z_k)$ has the local lifting property due to Boca \cite{Boca91} and the fact that minimal and maximal tensor products coincide each time we tensor C*-algebras with local lifting property and WEP \cite{Kirchberg93}. In summary, the above arguments imply the following.

\begin{theorem}
    Let $\cA \subseteq B(\mathcal{H})$ be C*-algebra. The following statements are equivalent.
    \begin{enumerate}
        \item $\cA$ has the WEP.
        \item $\cA$ has the complete tight $(n,k)$-Riesz interpolation property in $B(\mathcal{H})$ for all $n,k \in \mathbb{N}$.
        \item $\cA$ has the complete tight $(2,3)$-Riesz interpolation property in $B(\mathcal{H})$.
    \end{enumerate}
\end{theorem}

$ $

\bibliographystyle{alpha}

\begin{thebibliography}{}

\bibitem{AHMR88} J. Agler, J. W. Helton, S. McCullough, and L. Rodman, \emph{Positive semi-definite
matrices with a given sparsity pattern}, Linear Algebra and its Applications, {\bf 107}
(1988), 101–-149.

\bibitem{AKRSx} R. Araiza, L. Kroell, T. Russell and T. Sinclair \emph{Approximate factorization properties for operator systems} Preprint, arXiv:2607.00177.

\bibitem{Arv}   W. B. Arveson, {\em Subalgebras of C$^{*}$-algebras,}  Acta Math.\  {\bf 123} (1969), 141--224.  

\bibitem{CEinj} M.-D. Choi and E. G. Effros,  {\em Injectivity and operator spaces,} J.\ Funct.\ Anal.\  {\bf 24} (1977), 156--209.

\bibitem{Dirac61} G. A. Dirac, \emph{On rigid circuit graphs}, Abhandlungen aus dem Mathematischen Seminar der Universit¨at Hamburg, {\bf 25} (1961), 71–-76.

\bibitem{FP} D. Farenick and  V. I. Paulsen, {\em Operator system quotients of matrix algebras and their tensor products}, Math.\ Scand.\ {\bf 111} (2012), 210--243.

\bibitem{DymGohberg81} H. Dym and I. Gohberg, \emph{Extension of band matrices with band inverses}, Linear Algebra App. {\bf 36} (1981), l--24. 

\bibitem{Grone84} R. Grone, C. R. Johnson, E. M. S\'a, and H. Wolkowicz, {\em Positive definite completions of partial hermitian matrices}, Linear Algebra and its Applications, { \bf 58} (1984),
109–-124.

\bibitem{GuptaEtAl} Ved Prakash Gupta, Prabha Mandayam, V. S. Sunder, \emph{The Functional Analysis of Quantum Information Theory}, Lecture Notes in Physics 902, Springer, 2015.

\bibitem{HP} K. H. Han  and  V. I. Paulsen, {\em  An approximation theorem for nuclear operator systems}, Journal of Functional Analysis, {\bf 261}  (2011), 999--1009.

\bibitem{Harris} S. J. Harris, \emph{A Non-Commutative Unitary Analogue of Kirchberg's Conjecture}, Indiana University Mathematics Journal, {\bf 68(2)},(2019), 503–536. 

\bibitem{Kavruk} A. Kavruk, {\em Nuclearity related properties in operator systems}, Journal of Operator Theory, {\bf 71} (2014), 95--156. 

\bibitem{Kavrukw} A. Kavruk,  {\em The weak expectation property and Riesz interpolation}, Preprint (2012), arXiv:1201.5414

\bibitem{KavrukNP} A. Kavruk,  {\em On a non-commutative analogue of a classical result of Namioka and Phelps,} Journal of Functional Analysis, {\bf 269} no~10 (2015), 3282--3303.

\bibitem{KPTT1} A. Kavruk, V. I. Paulsen, I. G. Todorov and M. Tomforde, {\em Tensor products of operator systems}, Journal of Functional Analysis, {\bf 261} (2011), 267--299. 

\bibitem{KPTT2} A. Kavruk, V. I. Paulsen, I. G. Todorov and M. Tomforde, {\em Quotients, exactness, and nuclearity in the operator system category}, Advances in Mathematics, {\bf 235} (2013), 121--360.

\bibitem{Laurent01} M. Laurent, {\em On the sparsity order of a graph and its deficiency in chordality}, Combinatorica {\bf 21} no. 4 (2001), 543--570. 

\bibitem{McCullough88} S. McCullough, \emph{2-chordal graphs}, in: I. Gohberg, J. W. Helton and L. Rodman, eds., Contributions to Operator Theory and its Applications, Vol. 35 in Series \emph{Operator Theory: Advances and Applications}, pages 143–192, Birkhauser-Verlag, Basel, 1988.

\bibitem{McCullough93} S. McCullough, \emph{Minimal separators of 2-chordal graphs}, Linear Algebra and its Applications, 184 (1993), 187–199.

\bibitem{Oz} N. Ozawa, {\em About the Connes embedding conjecture--Algebraic approaches}, Jpn. J. Math. {\bf 8} (2013),  147–-183. 

\bibitem{Pnbook} V. I. Paulsen, {\em Completely bounded maps and operator
algebras,} Cambridge Studies in Advanced Math., 78, Cambridge
University Press, Cambridge, 2002.

\bibitem{PPS89} V.I. Paulsen, S. Power, R. Smith, \emph{Schur products and matrix completions}, Journal of Functional Analysis {\bf 85} (1989), 151--178.

\bibitem{PaulsenRodman} V. I. Paulsen and L. Rodman, {\em Positive completions of matrices over C*-algebras}, Journal of Operator Theory {\bf 25} no.~2 (1991), 237--253.

\bibitem{Boca91} F. Boca,
Free products of completely positive maps and spectral sets, 
J. Funct. Anal. 97 (1991), 251-263.


\bibitem{EffrosAspect} E. Effros, Aspects of non-commutative order,
in C*-algebras and applications to physics (Proc. 2nd Japan-USA Seminar, Los Angeles, 1977), eds.


\bibitem{GS}
I$.$ Goldbring, T$.$ Sinclair, Omitting types in operator systems, Indiana University Mathematics Journal
Vol$.$ 66, No$.$ 3 (2017), pp. 821-844

\bibitem{Kirchberg93} E. Kirchberg, On non-semisplit extensions, tensor products and exactness of group C*-algebras, {\it Invent. Math.} 112 (1993) 449 - 489.

\bibitem{LanceWEP} C. Lance, On nuclear C*-algebras, J. Funct. Anal. 12(1973), 157-176.


\bibitem{LupiniWEP} M. Lupini, An Intrinsic Order-Theoretic Characterization of the Weak Expectation Property, {\it Int$.$ Equations and Op$.$ Theory} Vol$.$ 90 (2018)

\bibitem{Pisier_intr}
G. Pisier, {\it Introduction to Operator Space Theory}, {\rm
Cambridge University Press, 2003}

\bibitem{Wittstock}
G. Wittstock, Ein operatorwertiger Hahn–Banach Satz., J. Funct. Anal. 40(2), 127-150 (1981)

\end{thebibliography}

\end{document}